\documentclass[12pt,a4paper]{article}

\usepackage{latexsym,amssymb,amsthm}
\usepackage{url,graphics}
\usepackage[dvips]{graphicx,epsfig}

\title{\bf Spanning trees with many leaves: new lower bounds in terms of number of vertices of degree~3 
and at least~4}
\author{
         D.\,V.\,Karpov\thanks{}
     \\[6pt]
          {\small E-mail: \texttt{dvk0@yandex.ru}       }
    }
\date{}

\begin{document}
\maketitle
\righthyphenmin=2
\renewcommand*{\proofname}{\bf Proof}
\newtheorem{thm}{Theorem}
\newtheorem{lem}{Lemma}
\newtheorem{cor}{Corollary}
\theoremstyle{definition}
\newtheorem{defin}{Definition}
\theoremstyle{remark}
\newtheorem{rem}{\bf Remark}

\def\N{{\rm N}}
\def\q#1.{{\bf #1.}}
\def\I{{\rm Int}}
\def\R{{\rm Bound}}
\def\mmin{\mathop{\rm min}}

\centerline{\sc Abstract}

We prove, that every connected graph with $s$ vertices of degree~3 and~$t$ vertices of degree at least~4
has a spanning tree  with   at least ${2\over 5}t +{1\over 5}s+\alpha$  leaves, where $\alpha \ge {8\over 5}$.
Moreover, $\alpha \ge 2$ for all graphs besides three exclusions. All exclusion are  regular graphs of degree~4, they are explicitly described in the paper.

We present infinite series of graphs, containing only vertices of degrees~3 and~4, for which the maximal number of leaves in a spanning tree is equal for ${2\over 5}t +{1\over 5}s+2$. Therefore we prove that our bound
is tight.

\section{\bf Introduction. Basic notations}

We consider unoriented graphs without loops and multiple edges.
We use standart notations. For a graph $G$ we denote the set of its vertices by  $V(G)$ and the set of its edges by $E(G)$. We use notations  $v(G)$ and $e(G)$ for the number of vertices and edges of $G$, respectively. 

We denote the {\it degree} of a vertex $x$ in the graph $G$ by $d_G(x)$.
For any set of vertices~$W\subset V(G)$ we denote by~$d_{G,W}(x)$ the number of vertices of the set~$W$, which are adjacent to~$x$ 
in the graph~$G$. As usual, we denote the minimal vertex degree of the graph~$G$ by~$\delta(G)$. 
 
Let $\N_G(x)$ denote the {\it neighborhood}  of a vertex  $w\in V(G)$ (i.e. the set of all vertices, adjacent to~$w$).

For any edge $e\in E(G)$ we denote by $G\cdot e$ the graph, in which the ends of the edge $e=xy$ are {\it contracted} into one vertex, which is incident to all edges, incident in~$G$ to at least one of the vertices~$x$ and~$y$. Let us say that the graph $G\cdot e$ is obtained from $G$ by {\it contracting} the edge $e$.

We call a set of vertices~$R\subset V(G)$ a {\it cutset} if the graph~$G-R$ is disconnected.

\begin{defin}
For any connected graph~$G$ we denote by~$u(G)$  the maximal number of leaves in a spanning tree of the graph~$G$. 
\end{defin}

\begin{rem}
Obviously, if  $F$ is a tree, then   $u(F)$ is the number of its leaves.
\end{rem}

Several papers about lower bounds on $u(G)$ are published.  
One can see details of the history of this question in~\cite{KB}. We shall recall only results, directly concerned with our work.

In 1981 Linial formulated a  conjecture:  
$$u(G)\ge {d-2\over d+1}v(G) + c \quad as \quad \delta(G)\ge d\ge 3,$$
 where a constant $c>0$ depends only on~$d$. The  ground  for this  conjecture is the  following:  for every  $d\ge 3$
one can easily construct infinite series of graphs with   minimal degree  $d$, for 
which  ${u(G)\over v(G)}$ tends to $d-2\over d+1$.

It follows from the works~\cite{Alon,DJ,YW} that for $d$ large enough  Linial's conjecture fails.  However, we are interested in the case of  small~$d$.

In 1991 Kleitman and West~\cite{KW} proved, that $u(G)\ge {1\over 4}\cdot v(G)+2$ as~$\delta(G)\ge 3$ and~$u(G)\ge {2\over 5}\cdot v(G)+{8\over 5}$ as~$\delta(G)\ge 4$.
In~1996 Griggs and Wu~\cite{JGM} once again proved the statement for~$\delta(G)\ge4$
and proved, that~$u(G)\ge {1\over 2}\cdot v(G)+2$ as~$\delta(G)\ge 5$. 
Hence, Linial's conjecture holds for~$d=3$, $d=4$ and~$d=5$, for~$d>5$ the question remains open.

In~\cite{KW} a more strong Linial's conjecture was mentioned: 
   $$u(G)\ge \sum_{x\in V(G)} {d_G(x)-2\over d_G(x)+1}$$ 
for a connected graph~$G$ with $\delta(G)\ge 2$. 
Clearly, this conjecture is not true, since weak Linial's conjecture fails for large degrees.
We present infinite series of connected graphs which vertices have degrees~3 and~4, disproving this conjecture. Therefore, the strong  conjecture fails not only for  huge degrees, coming to us from probabilistic
methods, but even for degrees~3 and~4.

However, strong Linial's  conjecture inspires attempts to obtain a lower bound on~$u(G)$, in which contribution of each vertex  depends on its degree. But how much must be  the contribution of a vertex of degree~$d$?

N.\,V.\,Gravin~\cite{Gr} proved for a connected graph with~$v_3$  vertices of degree~3 and~$v_4$ vertices of degree at least~4,  that  $u(G)\ge {2\over 5}\cdot v_4 +{2\over 15}\cdot v_3$. 
In this paper vertices of degrees~1 and~2 are allowed  in the graph.
There is no doubt that the constant~$2\over 5$ is optimal, but the constant~$2\over 15$ 
can be replaced by greater one, as it is shown in our main theorem.

\begin{thm}
\label{u34} 
Let $G$ be a connected graph with at least two vertices, 
$s$  is the number of vertices of degree~$3$, and~$t$ is the number of vertices of degree
at least~$4$ in~$G$. Then~$u(G)\ge {2\over 5 } t + {1\over 5} s +\alpha $, where~$\alpha\ge {8\over5}$.  
Moreover, $\alpha\ge 2$ for all graphs besides three exclusions: $C_6^2$, $C_8^2$ (squares of cycles on~$6$ and~$8$ vertices) and~$G_8$ --- a regular graph of degree~$4$ on~$8$ vertices, shown on figure~$1$. 
\end{thm}

\begin{figure}[!hb]
	\centering
		\includegraphics[width=1\columnwidth, keepaspectratio]{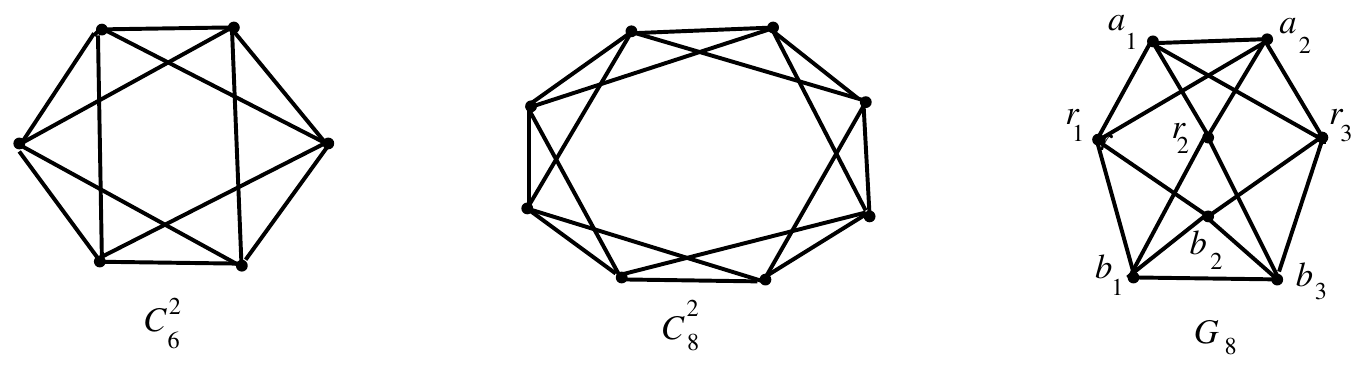}
     \caption{Graphs-exclusions.}
	\label{fig1}
\end{figure}

Note, that all three constants of this bound are optimal. There exist infinite series of examples, for which this bound is attained. We present series of such graphs, containing only vertices of degree~3 and~4. 

The proof of this theorem would be much shorter, if we exclude from the theorem the last statement.
Our interest to findning the exact additive constant is inspired by desire to obtain  a tight bound which is not a tip effect. The bound with~$\alpha={8\over 5}$ is attained for the only graph~$C_6^2$!
However, there are different infinite series of examples for~$\alpha=2$, 
and one can see  this additive constant in lower bounds for graphs with minimal degrees~3 or~5. 
Possibly, by this reason Kleitman and West~\cite{KW} have conjectured, that there are only two connected graphs with $\delta(G)\ge 4$ for which the bound~$u(G)\ge {2\over 5} v(G) +2$ fails:~$C_6^2$ and~$C_8^2$. Moreover,  
it is proved in~\cite{KW}, that for graphs with minimal degree~4 an exclusion must be a 4-regular graph and each its edge must belong to a triangle.

Kleitman and West  in their conjecture about graphs-exclusions didn't find only the graph~$G_8$ on eight vertices (see figure~1).
However, with the method from~\cite{KW} one cannot even prove, that the set of graphs-exclusions is finite.  
We shall prove in theorem~\ref{u34} the similar statement about graphs-exclusions for more general problem.

\section{Proof of theorem 1}

Let us introduce necessary notations.

\begin{defin}
Let $H$ be an arbitrary graph.
We denote by~$S(H)$ the set of all vertices of degree~3 of the graph~$H$, 
and by~$T(H)$ --- the set of all vertices of degree at least~4 of the graph~$H$. 

Let~$x\in V(H)$. We set that the  {\it cost} $c_H(x)$ of the vertex~$x$ in the graph~$H$~is
$$c_H(x) = \left\{ \begin{array}{ll}    {2\over 5} & \mbox{ as } x\in T(H), \\[3pt]
                                         {1\over 5}  & \mbox{ as } x\in S(H), \\ [1pt] 
                                          0          & \mbox{ as } x\notin T(H)\cup S(H). \\
                        \end{array} \right. $$
The {\it cost}  of the graph $H$ is
$$c(H)={2\over 5} |T(H)| +{1\over 5}|S(H)| =\sum_{x\in V(H)} c_H(x).$$ 
For any set of vertices~$U\subset V(H)$ we set, that the {\it cost} of this set in the graph~$H$ is~$c_H(U)=\sum_{x\in U} c_H(x)$. 
For  any tree~$F$, which is a subgraph of the graph~$H$, we set that its {\it cost in the graph~$H$}
is~$c_H(F)=c_H(V(F))$.

For any spanning tree~$F$ of the graph~$H$  we set the notation~$\alpha(F)=u(F)-c(H)$.
Let~$\alpha(H)$ be the maximum of~$\alpha(F)$ over all spanning trees~$F$ of the graph~$H$.
\end{defin}

\begin{rem}
It follows directly from the definition, that~$u(G)=c(G)+\alpha(G)$. Hence we want to prove, that~ $\alpha(G)\ge 2$  for almost all 
connected graphs~$G$.
\end{rem}

As usual,   during  construction of the desired spanning tree for a graph~$G$ we assume, that the theorem has been proved for all smaller graphs.

\subsection{Reduction rules}

At first we transform the graph such that it would be covinient to work with it. 
Let us describe two reduction rules.

\begin{figure}[!hb]
	\centering
		\includegraphics[width=1\columnwidth, keepaspectratio]{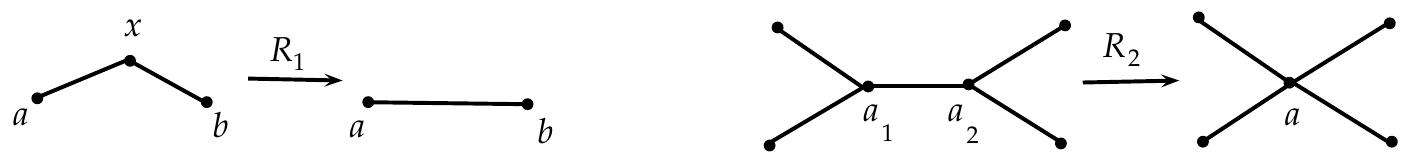}
     \caption{Reduction rules}
	\label{fig2}
\end{figure}

\smallskip
\q{R1}. {\it Let $x\in V(G)$, $d_G(x)=2$, $\N_G(x)=\{a,b\}$ and the vertices~$a$  and~$b$ are not adjacent. } 

\noindent We reduce the graph~$G$ to~$G'= G-x+ab$.  Obviously, $c(G')=c(G)$.

\smallskip

\q{R2}. {\it Let  $a_1,a_2\in S$ are adjacent vertices and $\N_G(a_1)\cap \N_G(a_2) =\varnothing$.}

\noindent We reduce the graph~$G$ to~$G'=G\cdot a_1a_2$. Let~$a$ be the vertex, obtained by gluing the vertices~$a_1$ and~$a_2$.
Clearly, $d_{G'}(a)=4$, then $c(G')=c(G)$.

\smallskip
In both cases one can easily transform  a spanning tree~$F'$ of the graph~$G'$ into a spanning tree~$F$ of the graph~$G$ with~$u(F)\ge u(F')$
and, therefore, $\alpha(F)\ge \alpha(F')$. Thus it is easy to see, that~$\alpha(G)\ge \alpha(G')$.

\begin{rem}
 \label {r1}
Later we may assume, that  considered graph satisfy the following conditions:

$1^\circ$  any  vertex of degree~2 forms a triangle with two vertices of its neighborhood;

$2^\circ$ for any two adjacent vertices of degree~3 their neighborhoods have nonempty intersection.
\end{rem}

\subsection{Dead vertices method. General description}

To prove the theorem we shall construct the desired spanning tree using the method of {\it dead vertices},
as in works~\cite{KW,JGM}.

\begin{defin}
Let a tree~$F$ be a subgraph of a connected graph~$G$.

We say that a leaf~$x$ of the tree~$F$ is  {\it dead}, if~$N_G(x)\subset V(F)$ and {\it alive}
otherwise. We denote by~$b(F)$ the number of dead leaves of the tree~$F$.

We set~$\alpha'(F)= {13\over 15}u(F) + {2\over15} b(F) -c_G(F)$.
\end{defin}

\begin{rem}
$1)$ It is easy to see, that dead leaves remain dead during all next steps of the construction. 
When the algorithm stops and we obtain a spanning tree, all its leaves will be dead.

$2)$ Since all leaves of a spanning tree are dead, we have~$\alpha'(F)=\alpha(F)$.
\end{rem}

We construct a spanning tree in~$G$  successively, adding the vertices in several steps.
Let~$S=S(G)$ and~$T=T(G)$.

Let us describe in details a step of our algorithm  (let's call this step~$A$).
Let we have a tree~$F$ before the step~$A$  (of course, $F$ is a subgraph of the graph~$G$).

We denote by~$\Delta u$ and~$\Delta b$ the {\it increase}  of the number of leaves and dead leaves in the tree~$F$, respectively,
on the  step~$A$, by~$\Delta t$  and~$\Delta s$ 
--- the number of added on this step  to the tree~$F$ vertices from~$T$ and~$S$, respectively.

We call by the {\it profit} of the step~$A$ the value
$$p(A)={13\over 15}\Delta u + {2\over15} \Delta b - {2\over5} \Delta t - {1\over5} \Delta s.$$

Let~$F_1$  be the tree obtained after the step~$A$. Clearly, $\alpha'(F_1)=\alpha'(F)+p(A)$.
We shall perform only steps with non-negative profit. 

At first we  describe all possible steps  and after that consider  beginning   of the construction
and estimate~$\alpha(T)$ for the constructed tree~$T$.

We denote by~$W$ the set of all vertices, which  do not belong to the tree~$F$.

Vertices of the set~$W$, which are adjacent to at least one vertex of the set~$V(F)$, are called 
{\it vertices of level}~1. Vertices of the set~$W$, which do not belong to level~1 and are adjacent to at least one vertex of level~1, are called  {\it vertices of level}~2.

For each vertex~$x\in W$  we denote by~$P(x)$ the set of all  adjacent to~$x$ vertices of the set~$V(F)$.

\subsection{A step of the algorithm}

We shall try to perform next step of the algorithm in the following way. We shall pass to the next variant of the step only when all 
previous variants are impossible.  We shall not note this during   description of  steps.
We begin with the step, which in fact is not a step, but will help us in  description of other steps.

\smallskip
\q{Z0}. {\it A leaf~$v$ of the tree~$T$, calculated as alive, appears dead.}

\noindent We do not transform the tree on this step. We take into account  information about~$v$ and obtain 
$$\Delta u=0, \quad \Delta b=1,\quad p(Z0)={2\over 15}.$$

\begin{rem}
 \label{rst}
During the description of steps we consider as alive  all leaves of the tree~$F$, which are not said to be dead.  Adding an extra dead vertex will be recorded as a step~$Z0$.
\end{rem}

Let us begin with some easy steps. At first four variants new leaves are added to the tree.

\smallskip
\q{A1}.  {\it There is a non-pendant vertex~$x$ of the tree~$F$, adjacent to~$y\in W$.} 

\noindent Then we adjoin~$y$ to~$x$. 
Clearly, $$\Delta u = 1, \quad  \Delta b =0, \quad p(A1) \ge {13\over 15}-{2\over 5} = {7\over 15}.$$

\smallskip \goodbreak
\q{A2}.  {\it There is a vertex   $x\in V(F)$ with~$d_{G,W}(x)\ge 2$.} 

\noindent Then we adjoin to the tree two adjacent to~$x$  vertices of the set~$W$. 
Clearly,
$$\Delta u = 1, \quad \Delta b=0, \quad  p(A2) \ge  {13\over 15} - 2\cdot{2\over 5} =  {1\over 15}. $$

\smallskip
\q{A3}.  {\it There is a vertex~$x$ of level~$1$, such that~$d_{G,W}(x)\ge 3$.}

\noindent
At first we adjoin to the tree~$F$ the vertex~$x$ and after that we adjoin three vertices of the set~$W$ adjacent to~$x$. 
The cost of four added vertices is not more than  $4\cdot{2\over 5}$, and we obtain
$$\Delta u=2, \quad \Delta b=0, \quad p(A3) \ge 2\cdot {13\over 15} - 4\cdot{2\over 5}   
= {2\over 15}. $$

\begin{figure}[!hb]
	\centering
		\includegraphics[width=0.8\columnwidth, keepaspectratio]{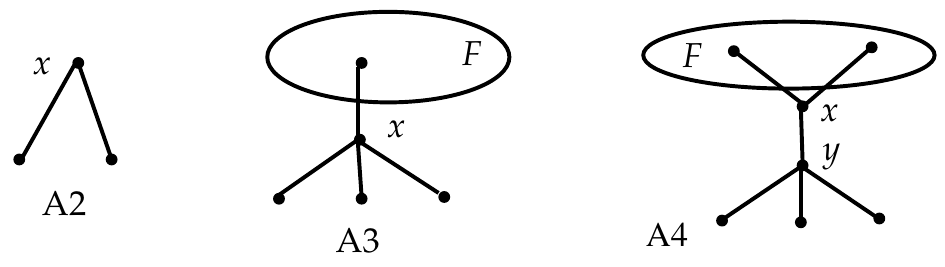}
     \caption{Steps of type~A.}
	\label{fig3}
\end{figure}

\begin{rem}
 \label{rl1}
Hereafter we assume that non-pendant vertices of the tree~$F$ are not adjacent to vertices of the set~$W$, 
any leaf of the tree~$F$ is adjacent to not more than one vertex of the set~$W$ and, finally,
any vertex of level~1 is adjacent to not more than two vertices of the set~$W$. 

In particular, if  $x\in T$  is a vertex of level~1, then~$|P(x)|\ge 2$ and after adjoining of the vertex~$x$ to the tree
at least one vertex of the set~$P(x)$ becomes a dead leaf of the obtained tree. 
\end{rem}

\q{A4}.  {\it There exists a vertex~$x\in S$ of level~$1$, adjacent to exactly one vertex of the set~$W$ --- a vertex~$y\in T$  of level~$2$.}

\noindent  
At first we adjoin to the tree~$F$ the vertices~$x$, $y$. Since $y$ is not adjacent to the tree~$F$,   there are three  vertices of the set~$W$, which are adjacent to~$y$ and different from~$x$.
We adjoin these three vertices to the tree. Then 
$$ \Delta u=2, \quad \Delta b =1, \quad   p(A4) \ge 2\cdot {13\over 15} +  {2\over 15}- {1\over 5} - 4\cdot{2\over 5} = {1 \over 15}. $$

\smallskip
Further on we consider a more complicated case. 

\smallskip

\q{M}.  {\it There is a vertex~$x\in T$ of level~$1$, such that~${d_{G,W}(x)=2}$.}

\noindent We adjoin the vertex~$x$ to the tree. Note, that~$c_G(x)={2\over 5}$. The vertex~$x$ is adjacent to at least two leaves
of the tree~$F$ (see figure~\ref{fig4}), hence at least one of these leaves becomes dead.
Then we adjoin to the tree two vertices~$y_1,y_2\in W$, adjacent to~$x$ (in fact, it is  the step~$A2$). 
Taking into account written above, we obtain
$$ \Delta u=1, \quad  \Delta b=1, \quad
p(M)\ge {2\over 15} - {2\over 5}+ p(A2) \ge -{3\over 15}.$$

\q{N}.  {\it There is a vertex~$x\in S$ of level~$1$, such that~${d_{G,W}(x)=2}$.}

\noindent We adjoin to the tree the vertex~$x$ and two vertices~$y_1,y_2\in W$, adjacent to~$x$.
Since~$c_G(x)={1\over 5}$, similarly to the previous case  we obtain
$$ \Delta u=1, \quad  \Delta b=0, \quad p(N)=  - {1\over 5}+ p(A2) \ge -{2\over 15}.$$

\smallskip
{\it We do not consider that steps~$M$ and~$N$ are finished. We have added to the tree three vertices~$x$, $y_1$, $y_2$. However, let~$F$ 
 is still the tree, constructed after previous  {\tt finished} step.
After performing  step~$M$ or~$N$ we have:
 
-- each leaf of the tree~$F$ is adjacent to not more than one vertex of the set~$W$; 

-- each vertex of level~$1$ is adjacent to not more than two vertices of the set~$W$.

We aim to perform a step with the profit at least~$3\over 15$.}

Let us continue case analysis.

\smallskip
\q{1}.   $y_1,y_2\not\in T$.

\noindent These vertices cost cheaper, than it was calculated above. Hence the profit increases by at least~$2\over 5$ and we obtain
$$\Delta u=\Delta b=0 , \quad p(1) \ge {6\over 15}.$$

\medskip \goodbreak
{\it We set~$W_1=W\setminus \{x,y_1,y_2\}$.

If only one of vertices~$y_1$ and~$y_2$ belongs to the set~$T$, we assume that~${y_1\in T}$.

If~$y_1, y_2\in T$, we assume, that $d_{G,W_1}(y_1)\ge d_{G,W_1}(y_2)$.
}
\smallskip

\begin{figure}[!hb]
	\centering
		\includegraphics[width=1\columnwidth, keepaspectratio]{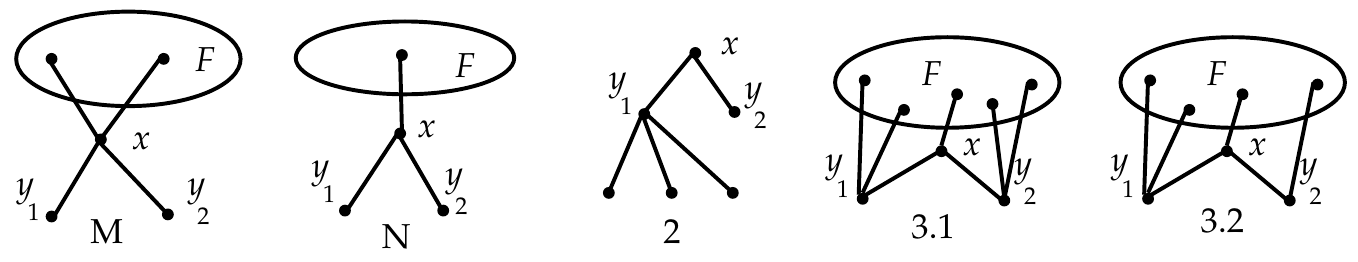}
     \caption{Steps $M$, $N$, 2 and 3}
	\label{fig4}
\end{figure}

 \goodbreak
\q{2}.  $d_{G,W_1}(y_1)\ge 3$.

\noindent 
We adjoin to the tree three vertices of the set~$W_1$, adjacent to~$y_1$ (in fact, we perform the steps~$A2$ and~$A1$). 
We obtain
$$\Delta u = 2,  \quad \Delta b = 0,\quad p(2) = p(A1) + p(A2) \ge {8\over 15}.$$

\smallskip \goodbreak
\q{3}.    $d_{G,W_1}(y_1)\le 1$.

\noindent The vertex~$y_1$ is adjacent to not more than three vertices of the set~$W$: they are~$x$ and, possibly, $y_2$ and one vertex of the set~$W_1$. Since~$d_G(y_1)\ge 4$, then  $y_1$ is adjacent to the tree~$F$, i.e. a vertex of level~1. By remark~\ref{rl1} we have~$|P(y_1)|\ge 2$, hence  the number of
dead leaves  increases by at least~2.  Consider two cases.

\q{3.1}. {\it If $y_2\in T$}, then by the choice of the vertex~$y_1$ we have~$d_{G,W_1}(y_2)\le 1$.
Similarly to written above for the vertex~$y_1$,  we have two additional dead leaves, adjacent to~$y_2$. In this case
$$\Delta u=0, \quad \Delta b=4,  \quad p(3.1) =  4\cdot {2\over 15} = {8\over 15}.$$

\q{3.2}. {\it If  $y_2\not\in T$}, then the vertex~$y_2$ costs  cheaper than we have calculated, hence the profit increases by~$1\over 5$ and we have
$$\Delta u=0,\quad \Delta b=2, \quad p(3.2) \ge {1\over 5} +  2\cdot {2\over 15} = {7\over 15}.$$

\smallskip
\q{4}.  $d_{G,W_1}(y_1)=2$.

\noindent Let~$z_1$ and~$z_2$  be two adjacent to~$y_1$ vertices of the set~$W_1$.
We adjoin~$z_1$ and~$z_2$ to the tree (it is a step~$A2$) and obtain
$$\Delta u=1,\quad \Delta b=0, \quad p(4) = p(A2)\ge {1\over 15},$$ that is not enough. Let us continue  case analysis.

\smallskip \goodbreak
\q{4.1}. {\it Among~$y_2,z_1,z_2$ there is a vertex adjacent to the tree~$F$. }
 
\noindent For example, let~$z_1$ be adjacent to the tree~$F$, i.e.~$z_1$ is a vertex of level~$1$.
For other vertices the reasoning is quite similar.

\q{4.1.1}. {\it If~$z_1\in T$}, then by Remark~\ref{rl1} the vertex~$z_1$ must be adjacent to at least two  leaves of the tree~$F$,  
hence the number of dead leaves increases by two and we obtain
$$\Delta u=1,\quad \Delta b=2, \quad p(4.1.1) \ge p(4) +  2\cdot {2\over 15}  \ge {5\over 15}.$$

\q{4.1.2}. {\it If~$z_1\not\in T$}, then the vertex~$z_1$ adds one dead leaf and the cost of~$z_1$ decreases by at least~$1\over 5$. 
Hence we obtain 
$$\Delta u=1,\quad \Delta b=1, \quad p(4.1.2) \ge p(4)+ {1\over 5} +  {2\over 15}\ge {6\over 15} .$$

\smallskip
{\it Further on we consider the case when the vertices~$y_2,z_1,z_2$ are not adjacent to the tree~$F$.}

\smallskip
\q{4.2}. {\it Among~$y_2,z_1,z_2$ there is a vertex not from the set~$T$. }

\noindent  This vertex increases profit by at least~${1\over 5}$ and we obtain
$$\Delta u=1, \quad \Delta b=0, \quad p(4.2)  \ge p(4) +  {1\over 5}  \ge {4\over 15}.$$

\begin{figure}[!hb]
	\centering
	\includegraphics[width=1\columnwidth, keepaspectratio]{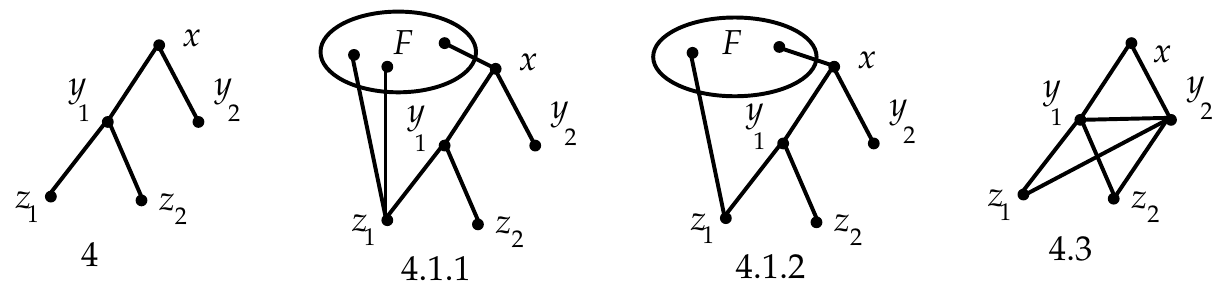}
     \caption{Steps 4, 4.1.1, 4.1.2 and 4.3}
	\label{fig5}
\end{figure}

\smallskip
\q{4.3}. {\it $\N_G(y_2)=\{x,y_1,z_1,z_2\}$. }

\noindent
In this case the vertex~$y_2$ is a dead leaf and we obtain
$$\Delta u=1,\quad \Delta b=1, \quad p(4.3)  \ge p(4) + {2\over 15} = {3\over 15}.$$

\begin{rem}
\label{rm33}
Let us summarize the analyzed cases. 
In remaining cases the vertices~$y_1,y_2,z_1,z_2$ belong to the set~$T$ and to the level~2.
Moreover, $d_{G,W_1}(y_1)=d_{G,W_1}(y_2)=2$, hence,  the vertices~$y_1$ and~$y_2$ are adjacent and~$d_G(y_1)=d_G(y_2)=4$.

The vertex~$y_2$ cannot be  adjacent to  both vertices~$z_1$ and~$z_2$. Without loss of generality we assume, that~$y_2$ is not adjacent to~$z_1$. Then the vertex~$z_1$  is adjacent to at least two vertices of the set~$W_2=W\setminus \{x,y_1,y_2,z_1,z_2\}$.  
\end{rem}

\q{4.4}. $d_{G,W_2}(z_1)\ge 3$.

\noindent We adjoin to the tree three vertices of the set~$W_2$ adjacent to~$z_1$ (in fact, we perform a step~$A2$ and a step~$A1$). 
We obtain $$\Delta u=3, \quad \Delta b=0, \quad p(4.4) \ge p(4)+ p(A2)+p(A1) \ge {9\over 15}.$$

\smallskip 
\q{4.5}. $d_{G,W_2}(z_1)=2$.

\noindent Denote by~$p_1$ and~$p_2$ two adjacent to~$z_1$ vertices of the set~$W_2$ and adjoin these two vertices to the tree (see.\,figure~6).
We have performed a step~$A2$ and obtain 
$$p(4.5) \ge p(4) +  p(A2) \ge {2\over 15}.$$ 
This is not enough for us, let's continue case analysis.

\smallskip
\q{4.5.1}. {\it Among~$p_1,p_2$ there is a vertex  of the set~$T$, which is adjacent with the tree~$F$.}

\noindent
Let it be~$p_1$. By Remark~\ref{rl1} the vertex~$p_1$ is adjacent with at least two leaves of the tree~$F$, hence, the number of dead leaves increases by at least~2 and we obtain
$$\Delta u=2, \quad \Delta b=2, \quad p(4.5.1) \ge  p(4.5) +  2\cdot {2\over 15}  \ge {6\over 15}.$$

\begin{figure}[!hb]
	\centering
		\includegraphics[width=0.9\columnwidth, keepaspectratio]{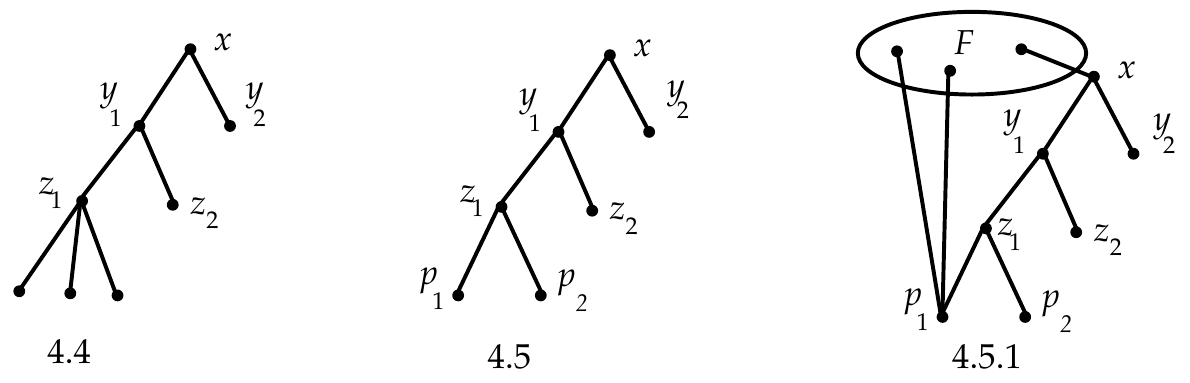}
     \caption{Steps 4.4, 4.5, 4.5.1}
	\label{fig6}
\end{figure}

\smallskip \goodbreak
\q{4.5.2}. {\it  Among~$p_1,p_2$  there is a vertex not from the set~$T$.}

\noindent
In this case the profit increases by at least~${1\over 5}$,  hence, we obtain
$$\Delta u=2, \quad \Delta b=0, \quad p(4.5.2) \ge p(4.5) +  {3\over 15}  \ge {5\over 15}.$$

\smallskip

\q{4.5.3}. {\it Among~$y_2,z_2,p_1,p_2$ there is a vertex, which is not adjacent to any vertex of the set~$W_3=W\setminus \{x,y_1,y_2,z_1,z_2,p_1,p_2\}$.}

\noindent This vertex is an extra dead leaf of the constructed tree, hence, the profit increases by at least~${2\over 15}$ and we obtain
$$\Delta u=2, \quad \Delta b=1, \quad p(4.5.3) \ge p(4.5) +  {2\over 15}  \ge {4\over 15}.$$

\begin{rem}
\label{rm352}
 Therefore, all vertices $y_1,y_2,z_1,z_2,p_1,p_2$ belong to the set~$T$ and are not adjacent to the tree~$F$. 
Each of the vertices~$y_2,z_2,p_1,p_2$ is adjacent to at least one vertex of the set~$W_3$.
\end{rem}

\smallskip \goodbreak
\q{4.5.4}.  $d_{G,W_3}(p_1)\ge 2$ {\it or } $d_{G,W_3}(p_2)\ge 2$.

\noindent  Without loss of generality we assume, that~$d_{G,W_3}(p_1)\ge 2$. We adjoin to the tree  two vertices~$q_1,q_2\in W_3$, adjacent 
to~$p_1$ (i.e. we perform a step~$A2$).  We obtain 
$$\Delta u=3, \quad  \Delta b=0, \quad  p(4.5.4) \ge p(4.5) +  p(A2)  \ge {3\over 15}.$$

\smallskip \goodbreak
\q{4.5.5}. {\it $d_{G,W_3}(p_1)=d_{G,W_3}(p_2)= 1$.}

\noindent  
Then the vertex~$p_1$ is adjacent to at least two of vertices~$p_2,y_2,z_2$, and the vertex~$p_2$ is adjacent to at least two of vertices~$p_1,y_2,z_2$. Let us remind, that by remark~\ref{rm33} the vertices~$y_1$ and~$y_2$ are adjacent and~$d_G(y_1)=d_G(y_2)=4$.
Since~$y_2$ is adjacent to at least one vertex from~$W_3$, then~$y_2$ cannot be adjacent to both vertices~$p_1$ and~$p_2$.
Without loss of generality we assume, that~$y_2$ is not adjacent to~$p_1$. Then~$d_G(p_1)=4$, and the vertex~$p_1$ must be adjacent
to both vertices~$p_2$ and~$z_2$ (see figure~\ref{fig7}a).  

Note, that the vertex~$z_2$ cannot be adjacent to~$p_2$. (Otherwise~$z_2$ would be adjacent to three vertices of the set~$W\setminus\{x,y_1,y_2,z_1,z_2\}$: they are~$p_1$, $p_2$ and a vertex from~$W_3$. In this case we adjoin these three vertices to~$z_2$
i.e.  perform a step~4.4 with~$z_2$ instead of~$z_1$.) Thus the vertex~$p_2$ is adjacent to~$y_2$ and~$d_G(p_2)=4$ (see figure~\ref{fig7}b).

Moreover,~$y_2$ is adjacent to exactly two vertices of the set~$W_1=W\setminus \{x,y_1,y_2\}$ by Remark~\ref{rm33}. Since~$y_2$ is adjacent to~$p_2$ and to a vertex from~$W_3$, it is  adjacent to neither~$z_1$ nor~$z_2$. Hence~$z_1$ is adjacent to~$z_2$  and~$d_G(z_1)=d_G(z_2)=4$ (see figure~\ref{fig7}c).

\begin{figure}[!hb]
	\centering
		\includegraphics[width=\columnwidth, keepaspectratio]{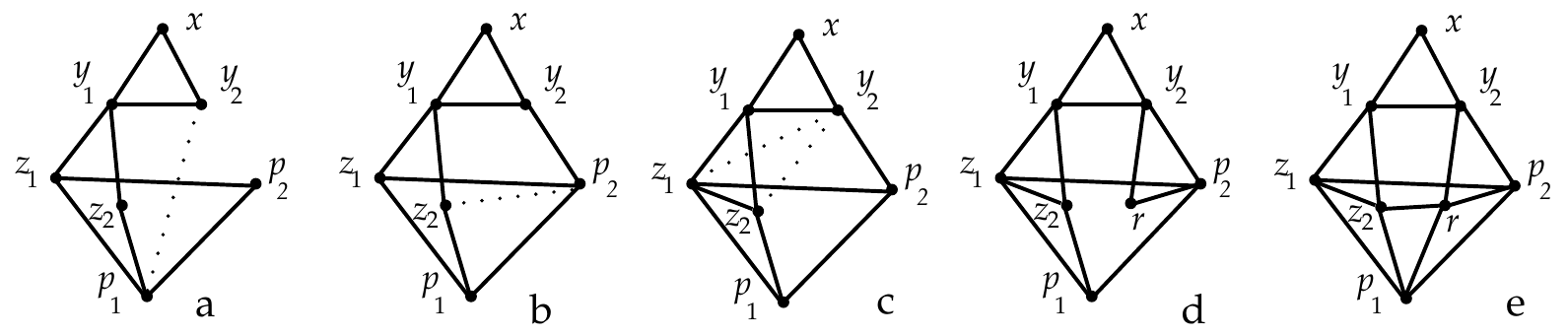}
     \caption{Step  4.5.5}
	\label{fig7}
\end{figure}

Denote by~$r$ the only vertex of the set~$W_3$, which is adjacent to~$y_2$.
Now one can apply to~$y_2$ the same  reasoning  as written above for~$y_1$ and obtain, that two adjacent to~$y_2$ vertices~$r$ and~$p_2$
are adjacent to each other and, in addition,~$d_G(r)=4$ (see figure~\ref{fig7}d). 

Continuing this reasoning for the vertex~$p_2$  and two vertices~$p_1$ and~$z_1$, adjacent to~$p_2$, we make sure that one of the vertices~$p_1$ and~$z_1$ must be adjacent to~$r$.  Since~$z_1$ cannot be adjacent to~$r$, then the vertices~$p_1$ and~$r$ are adjacent. 

Now it is clear (see figure~\ref{fig7}d), that~$z_2$ is adjacent to exactly two vertices of the set~$W_2$: these vertices are~$p_1$ and a vertex~$r'\in W_3$.
In this case one can repeat the reasoning written above for the vertex~$z_2$ instead of~$z_1$ and obtain, that $p_1$ ia adjacent to~$r'$. Hence,~$r=r'$ and~$z_2$ is adjacent to~$r$. We obtain the configuration, shown on figure~\ref{fig7}e.

We add to the tree the vertex~$r$ (adjoin it to any vertex adjacent to~$r$). Note, that now no added vertex is adjacent to a vertex outside the constructed tree. Let us calculate the parameters of this step from the very beginning: $\Delta t=5$,
$$\Delta u=2, \quad  \Delta b=4, \quad p(4.5.5) \ge 2\cdot {13\over 15} +4\cdot {2\over 15} - 5\cdot {2\over 5}= {4\over 15}.$$

\begin{rem}
\label{rm354}
1) We have proved that we can  perform after each of the steps~$M$ and~$N$ a step  with  profit 
at least~$3\over 15$. We shall always perform after steps~$M$ and $N$ one of the steps introduced above  to obtain a resulting step with non-negative profit. The notaition~$M4.2$ will mean a step, consisting of~$M$ and~$4.2$ after it. Similarly for  other steps.
We call such step by {\it $M$-step}, if the first step was~$M$ and {\it $N$-step} if the first step was~$N$. 
We call all these steps by {\it $MN$-steps}.

The profit of the steps~$M4.3$ and~$M4.5.4$ can be equal to zero, for all other $MN$-steps the profit is at least~${1\over 15}$. 
Any $MN$-step, except~$M4.5.5$ and~$N4.5.5$, cannot be the last step of the algorithm, since on this step we add at least one alive leaf.

2) Let us summarize the analyzed cases.  In remaining cases any vertex of level~1 is adjacent to at most one vertex of the set~$W$ 
(otherwise we can perform the step~$A3$ or one of $MN$-steps).
\end{rem}

In the next cases the number of leaves of the constructed tree does not vary, but the number of dead leaves increases.

\smallskip
\q{Z1}. { \it There exists a vertex of level~$1$, which is not adjacent to~$W$.}

\noindent Let it be a vertex~$w$. Then~$\N_G(w)=P(w)$. We add the vertex~$w$ to the tree. The vertex~$w$ and all vertices of the set~$\N_G(w)$, except one, become dead leaves of the new tree.  Therefore~$\Delta b = d_G(w)$. 

\q{Z1.1}.  $w\in T$.

\noindent  In this case the number of dead leaves increases by~$d_G(w) \ge 4$. We  consider that $\Delta b=4$, and if really $d_G(w)>4$
we record this as~$d_G(w)-4$ additional steps~$Z0$. Thus for the step~$Z1.1$ we have
$$\Delta u=0,\quad \Delta b= 4, \quad p(Z1.1)=4\cdot {2\over 15} - {2\over 5}= {2\over 15}.$$

\q{Z1.2}.  $w\in S$. 

\noindent In this case the parameters of the step are
$$\Delta u=0, \quad \Delta b= 3, \quad p(Z1.2)=3\cdot {2\over 15} - {1\over 5}= {3\over 15}.$$

\goodbreak
\q{Z1.3}.  $w\not\in S\cup T$. 

\noindent  In this case we have~$p(Z1.3)=\Delta b \cdot {2\over 15}$. We will not consider this step further on, since the parameters of this step are the same as the parameters of~$\Delta b$ consecutive steps~$Z0$.

\smallskip
\q{Z2}. { \it There are two adjacent vertices~$v,w$ of level~$1$.}

\noindent Let~$v,w$ be these vertices.  By remark~\ref{rm354} other vertices, adjacent to~$\{v,w\}$ are leaves of the tree~$F$.  Clearly, $d_G(v)\ge 2$ and~$d_G(w)\ge 2$. Let~$d_G(v)=2$ and~$\N_G(v)=\{x,w\}$. Then the~$x$, adjacent to~$w$ (otherwise we would apply reduction rule~$R1$). Hence we have~$d_{G,W}(x)\ge 2$ for a leaf~$x$ of the tree~$F$, that contradicts Remark~\ref{rl1}.

Thus, we have~$v,w\in S\cup T$.  The case~$v,w\in S$ is impossible (in this case we would apply reduction rule~$R2$).
We add the vertices~$v$ and~$w$ to the tree. In the obtained tree the vertices~$v$, $w$  and all vertices of~$P(v)\cup P(w)$, except 
two vertices, are dead leaves. Therefore,~$\Delta b = d_G(w)+d_G(v)-2$. 
As in step~$Z1.1$, further on we shall write minimal possible~$\Delta b$, and use, if it is necessary, steps~$Z0$.

\q{Z2.1}. {\it If one of vertices~$v,w$ belongs to~$S$ and another belongs to~$T$}, then
$$\Delta u=0, \quad \Delta b=5, \quad  p(Z2.1)=5 \cdot {2\over 15} -{1\over 5}-{2\over 5}= {1\over 15}.$$

\q{Z2.2}. {\it If $v,w\in T$}, then
$$\Delta u=0, \quad \Delta b=6, \quad p(Z2.2)=6\cdot {2\over 15} -2\cdot {2\over 5}= 0.$$

\smallskip
\q{Z3}. { \it There is a vertex~$w$ of level $1$ adjacent to a vertex~$v\in W \setminus  (S\cup T)$.}

\noindent We add the vertices~$w$  and~$v$ to the tree.  If $d_G(v)=2$ then $\N_G(v)\subset \N_G(w)$ (otherwise we would apply reduction rule~$R1$). Hence, in the obtained tree the vertex~$v$  and all vertices of~$P(w)$, except 
one, are dead leaves. We have $\Delta b = d_G(w)-1$. 

\q{Z3.1}. {\it If $w\in S$ } we obtain
$$\Delta u=0, \quad \Delta b=2, \quad p(Z3.1)=2\cdot {2\over 15} - {1\over 5}= {1\over 15}.$$

\q{Z3.2}. {\it If  $w\in T$}  we obtain
$$\Delta u=0, \quad \Delta b=3, \quad p(Z3.2)= 3\cdot {2\over 15} - {2\over 5} =  0.$$

\begin{lem}
 \label{lrz3}
Let~$w$ be a vertex of level~$1$. Then~$w\in T$, moreover, the vertex~$w$ is adjacent to  a vertex~$v\in S\cup T$ of level~$2$ and to at least three leaves of the tree~$F$.

\begin{proof}
By remark~\ref{rm354} we have~$d_{G,W}(w)\le 1$. Since we cannot perform the step~$Z1$, then~$d_{G,W}(w)= 1$, i.e.~$w$ is adjacent to a vertex~$v\in W$.

Since we cannot perform the step~$Z2$, then~$v$ is a vertex of level~2. Since we cannot perform the step~$Z3$, then~$v\in T\cup S$.
Since we cannot apply the reduction rule~$R1$, then~$w\in T\cup S$. 

Finally, let's prove, that~$w\in T$. Assume the contrary, let~${w\in S}$.  If~$v\in T$, we can perform a step~$A4$.
In the case~$v\in S$ we can apply the reduction rule~$R2$.  In both cases we have a contradiction.

Thus,~$w\in T$. Since~$d_{G,W}(w)= 1$, the vertex~$w$ is adjacent to at least three leaves of the tree~$F$.
 \end{proof}
\end{lem}

\begin{figure}[!hb]
	\centering
		\includegraphics[width=1\columnwidth, keepaspectratio]{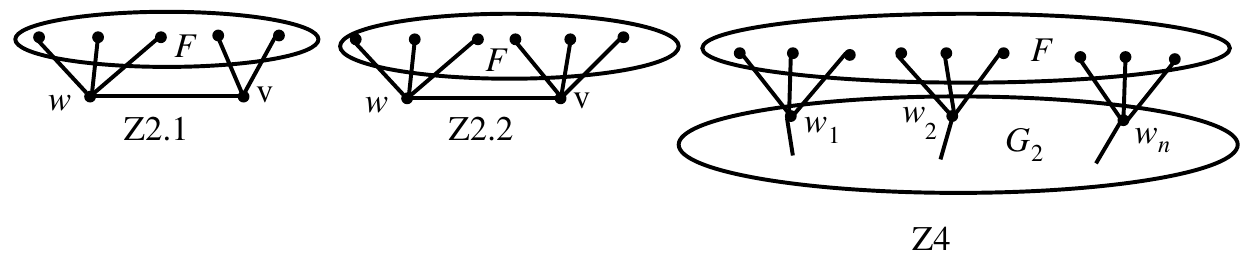}
     \caption{Steps of type $Z$.}
	\label{fig8}
\end{figure} 

\smallskip
\q{Z4}. { \it There exists a vertex in~$W$, i.e.~$F$ is not a spanning tree.}

\noindent Let~$w_1,\dots, w_n$ are all vertices of level~1. By lemma~\ref{lrz3} each of these vertices is adjacent with at least three leaves   of the tree~$F$. Thus there are at least~$3n$ alive leaves in the tree~$F$, i.e., ${u(F)-b(F)\ge 3n}$ and 
$$u(F)=c_G(F)+\alpha'(F) +{2\over 15}(u(F)-b(F)) \ge c_G(F)+\alpha'(F)+ {2n\over 5}.\eqno(1)$$

We delete all edges, connecting~$w_1,\dots,w_n$ with the tree~$F$. As a result the graph~$G$ will be splitted into~$G_1=G(V(F))$ and~$G_2=G(W)$. Since~$c_G(w_i)-c_{G_2}(w_i)={2\over 5}$, then $$c(G_2)= c_G(W)-n\cdot {2\over 5}. \eqno(2)$$

Note, that the graph~$G_2$ can be disconnected, but each its connected component contains at least 4 vertices and there is a pendant vertex
among them (one of the vertices~$w_1, \dots, w_n$). Hence, every connected component of~$G_2$ is not a graph-exclusion and contains less vertices than the graph~$G$. Thus we can apply the statement of our theorem to any connected component~$H$ of the graph~$G_2$ and construct
a spanning tree~$F_H$ in~$H$ with~${u(F_H)\ge c(H)+2}$. Therefore we can construct a forest~$F'$ in the graph~$G_2$, which consists of~$k$ spanning trees of connected components of~$G_2$. Clearly,
$${u(F')\ge c(G_2)+2k}. \eqno(3)$$ 

We adjoin each of~$k$ connected components of the forest~$F'$ to the tree~$F$ and obtain the spanning tree~$T$ of the graph~$G$. Let us estimate~$u(T)$ with the help of inequalities~$(1)$, $(3)$ and~$(2)$:
$$u(T)=u(F)+u(F')-2k \ge c_G(F)+\alpha'(F)+ {2n\over 5}+c(G_2)=$$
$$= c_G(V(F))+c_G(W) +\alpha'(F) =c(G)+\alpha'(F).$$
Thus, in this case we have~$\alpha(G)\ge \alpha(T)\ge \alpha'(F)$.

\subsection{Beginning of the construction and estimation of~$\alpha$}

We shall begin construction of the spanning tree with a {\it base tree}~$F'$ such that~$\alpha'(F')$ is  rather big.

We consider several cases. We shall pass to the next variant of the base only when all 
previous variants are impossible. We begin with the cases when one can easily construct a base tree~$F'$ with~$\alpha'(F')\ge 2$
and, hence, finish the proof of the theorem.

\smallskip
\q{B1}. {\it There are two adjacent vertices~$a,a'\in T$ with~$\N_G(a)\cap \N_G(a')=\varnothing$.}

\noindent We begin with the base tree~$F'$ in which the vertices~$a$ and~$a'$ are adjacent to each other and to all vertices from~$\N_G(a)\cap \N_G(a')$. Clearly,~$u(F')=u\ge 6$, $c_G(F')\le {2\over 5}(u+2)$ and 
$\alpha'(F')\ge  {13\over 15} u - c_G(F') \ge {7u-12 \over 15} \ge 2. $

\smallskip
\q{B2}. {\it There is a vertex~$a\in T$, adjacent to a vertex of degree not more than~$2$.}

\noindent Let~$v\in \N_G(a)$, $d_G(v)\le 2$. We begin with the base tree~$F'$, in which the vertex~$a$ is adjacent to all vertices from~$\N_G(a)$. In the case~$d_G(v)=1$ it is clear, that~$v$ is a dead leaf of~$F'$. Let~$d_G(v)=2$. Since we cannot apply the reduction rule~$R1$, then the vertices~$a$,~$v$ and some vertex from~$\N_G(a)$  form a triangle. Hence, in this case the vertex~$v$ is a dead leaf of~$F'$, too.

Therefore,  $u(F')=d_G(a)=u \ge 4$, $b(F')\ge 1$,  $c_G(F')\le {2\over 5}u$ and 
$$\alpha'(F') \ge  {13\over 15} u+ {2\over 15} -  c_G(F') \ge  {7u+2 \over 15} \ge 2. $$

\smallskip
{\it Further on we shall consider base trees~$F'$ with~$\alpha'(F')<2$. To provide~$\alpha(G)\ge 2$ we draw attention to the end of construction.}

\begin{lem}
\label{lf115}
Assume, that the graph~$G$ does not contain configurations, descripted in the cases~$B1$ and~$B2$ and one have constructed a spanning tree
with the help of our algorithm.   Then the following statements hold.

$1)$  If the step~$Z4$ was performed, then there exists a  base tree~$F'$  with ${\alpha'(F')\ge 2}$.

$2)$  If the step~$Z4$ was not performed, then  the last step of the algorithm do not add new alive leaves and has  profit at least~${1\over 15}$.

\begin{proof}
1) Let us return to the step~$Z4$ and to the graph~$G_2$, cut from the tree~$F$ (see figure~\ref{fig8}). We consider a connected component~$G'$ of the graph~$G_2$. Without loss of generality assume, that~$w_1,\dots,w_k \in V(G')$, $w_{k+1},\dots,w_n\notin V(G')$. Since~$G'$ is a smaller connected graph with pendant vertices, one can construct in~$G'$ a spanning tree~$T'$ with~$\alpha(T')=u(T')-c_{G'}(T')\ge 2$. 

Consider the tree~$T'$ as a subgraph of the graph~$G$. Unfortunately, each of the vertices~$w_1,\dots,w_k$ costs ${2\over 5}$ in~$G$
(while it costs~0 in~$G'$), hence, $c_G(T')=c_{G'}(T')+{2k\over 5}$. In addition, the vertices~$w_1,\dots,w_k$ are alive  leaves of~$T'$ in the graph~$G$ (all other leaves of the tree~$T'$ are, clearly, dead), hence, in the graph~$G$ we have $u(T')-b(T') = k$ and
$$\alpha'(T') = u(T')-c_G(T') -{2\over 15}\cdot (u(T')-b(T'))= u(T')-c_{G'}(T')-{8k\over 15} = 2-{8k\over 15}.$$

Lel us remember  details of the step~$Z4$ and for all~$i\in\{1,\dots, k\}$ consider three adjacent to~$w_i$ vertices $x^i_1,x^i_2,x^i_3\in V(F)$. All~$3k$ vertices in such triples are different and do not belong to the tree~$T'$. 
For each~$i\in \{1,\dots, k\}$ we adjoin to~$w_i$ three vertices~$x^i_1,x^i_2,x^i_3$. That is we~$k$ times perform  a step~$A2$ and  a step~$A1$.  The sum of profits is equal to~$k\cdot {8\over 15}$ and we obtain as a result the base tree~$F'$ with~$\alpha'(F')\ge  2$.

2) Consider the last step. No new alive leaves were added on this step. Looking over  parameters of the steps, one can make a conclusion, that
only one of the steps~$Z0$, $Z1.1$, $Z1.2$,  $Z2.1$, $Z2.2$,  $Z3.1$, $Z3.2$, $N4.5.5$ and~$M4.5.5$ can be the last. 
The step~$Z2.2$ is impossible,  since it needs the configuration from the case~$B1$.
The step~$Z3.2$ is impossible,  since it needs the configuration from the case~$B2$.
Each of other steps has profit at least~${1\over 15}$.
\end{proof}
\end{lem}

Thus hereafter it is enough to prove, that on the steps which add new alive leaves  a tree~$F$ with~$\alpha'(F)\ge {29\over 15}$ will be constructed. Let us continue the case analysis.

\smallskip
\q{B3}. {\it There is a vertex~$a$ of degree at least~$5$.}

\noindent  We begin with base the tree~$F'$, in which the vertex~$a$ is adjacent to all vertices from~$\N_G(a)$. Clearly,  $u(F')=d_G(a)=u \ge 5$,  $c_G(F')\le {2\over 5}(u+1)$ and $\alpha'(F')\ge {13\over 15}u -  c_G(F') \ge  {7u-6 \over 15} \ge {29\over 15}, $
that is enough.

\smallskip
\q{B4}. {\it There is a vertex~$x\in S$, adjacent to a vertex of degree not more than~$2$.}

\noindent  Let $v\in \N_G(x)$, $d_G(v)\le 2$, $\N_G(x)=\{v,y_1,y_2\}$. 
We begin with the  base tree~$F'$, in which the vertex~$x$ is adjacent to all vertices from~$\N_G(x)$.
Similarly to the case~$B2$, the vertex~$v$ is a dead leaf of~$F'$. Hence,  $u(F')=3$, $b(F')\ge 1$. 
We have $$\alpha'(F')\ge  {13\over 15}\cdot 3  +{2\over 15} - c_G(F') = {41 \over 15} - c_G(F').$$

If at least one of the vertices~$y_1,y_2$ does not belong to~$T$, then $c_G(F')\le 2\cdot {1\over 5}+  {2\over 5}={4\over5}$ and
$\alpha'(F')\ge {29 \over 15}. $  By lemma~\ref{lf115} we have~$\alpha(G)\ge 2$.

Consider the case~$y_1,y_2\in T$. Then both~$y_1,y_2$ are alive leaves,  $c_G(F')={1\over 5}+ 2\cdot {2\over 5}=1$  
and~$\alpha'(F')=  {26 \over 15}$. The construction is not finished, we need additional profit~${4\over 15}$. 

During analysis of the cases~$M$ and $N$ we have  cosidered a similar problem of deficiency of profit~${3\over 15}$. We repeat the same reasoning, perform a~$MN$-step which is possible in our configuration and obtain a tree~$F^*$ with~$\alpha'(F^*)\ge{29\over 15}$.
Moreover, we have $\alpha'(F^*)<2$ only for steps~$M4.3$ and~$M4.5.4$, but in both these cases the constructed trees have alive leaves. Hence,
the construction is not finished and by lemma~\ref{lf115} the last step will give us the profit at least~$1\over 15$ and 
provide~$\alpha(G)\ge 2$.

\begin{rem}
 In the items~$B2$ and~$B4$ we considered all cases when the graph contains a vertex of degree not more than~2.
In the item~$B3$ the case when the graph contains a vertex of degree more than~4 was considered.
{\it Hence further on we consider only  graphs with vertex degrees equal to~$3$ or~$4$.}
\end{rem}

\smallskip
{\it In table~$1$ we introduce parameters of all possible steps.}  The profits of all steps are multiplied by~15.

\smallskip
We have taken into account that the steps~$Z2.2$,  $Z3.1$, $Z3.2$  and~$Z4$  are impossible.
(For steps~$Z2.2$,  $Z3.2$ and~$Z4$  see details in Lemma~\ref{lf115} and its proof. The step~$Z3.1$ is impossible, since vertex degrees of our graph are at least~3.)

It is not convinient to deal with such a great number of steps. Next lemma will significantly decrease the number of possible steps.

\begin{lem}
\label {lmb} 
Assume, that  one have constructed a spanning tree in the graph~$G$ with the help of our algorithm.  
If it was performed one of $MN$-steps, mentioned below, then~$\alpha(G)\ge 2$.

$1)$ One of the  steps~$N4.2$, $N4.3$, $N4.4$, $N4.5.2$, $N4.5.3$, $N4.5.5$. One of the steps~$N1$, $N2$,  
$N4.5.4$, on which all added vertices, except~$x$,  were not adjacent to the tree~$F$.

$2)$  One of the  steps~$M4.2$, $M4.3$, $M4.4$, $M4.5.2$, $M4.5.3$, $M4.5.5$.  One of the  steps~$M1$, $M2$,  
$M4.5.4$, on which all added vertices, except~$x$,  were not adjacent to the tree~$F$. 

\end{lem}

\begin{proof} Let~$F$ be a tree, constructed before the mentioned $MN$-step. Remember detail of $MN$-steps: on this step we adjoin to~$F$ a tree~$F_0$ which root is a  vertex~$x\in S\cup T$ of level~1. Let~$p$ be the profit of this step. Note, that in all mentioned steps the vertex~$x$ is adjacent to exactly two vertices of the set~$W=V(G)\setminus V(F)$, and all vertices of the adjoint tree~$F_0$, except~$x$, 
are not adjacent to~$V(F)$. To make sure of these facts one can look over details of mentioned steps and take into account the conditions of lemma.

\medskip

{
\begin{tabular}{|@{\quad}c@{\quad}|@{\quad}c@{\quad}|@{\quad}c@{\quad}|}
\hline
Step &   $\Delta u-\Delta b$ &  $15\cdot$profit  \\ [2pt]
\hline
$A1$ &   1 &  $7$  \\
\hline
$A2$, \quad $A4$, \quad $M4.2$, \quad $N4.3$, \quad $M4.5.3$  &   1 & 1  \\
\hline
$A3$,\quad $N4.2$, \quad $M4.5.2$, \quad $N4.5.3$   &   2 &  2  \\
\hline
$M1$, \quad $M4.1.2$, \quad $M4.5.1$, \quad $N4.1.1$  &   0 &  3  \\
\hline
$N1$, \quad $N4.1.2$, \quad $N4.5.1$     &   1 &  4  \\
\hline
$M2$   &   2 &  5  \\
\hline
$N2$, \quad $M4.4$  &   3 &  6  \\
\hline
$M3.1$   &   $-4$ &  $5$  \\
\hline
$N3.1$   &   $-3$ &  $6$  \\
\hline
$M3.2$   &   $-2$ &  $4$  \\
\hline
$N3.2$   &   $-1$ &  $5$  \\
\hline
$M4.1.1$, \quad $N4.5.5$,\quad  $Z0$  &   $-1$ &  $2$  \\
\hline
$M4.3$  &   0 &  0  \\
\hline
$N4.4$  &   4 &  7  \\
\hline
$N4.5.2$  &   3 &  3  \\
\hline
$M4.5.4$  &   3 &  0  \\
\hline
$N4.5.4$  &   4 &  1  \\
\hline
$M4.5.5$  &   $-2$ &  $1$  \\
\hline
$Z1.1$   &   $-4$  &  $2$  \\
\hline
$Z1.2$  &   $-3$  &  $3$  \\
\hline
$Z2.1$    &   $-5$  &  $1$  \\
\hline
\end{tabular}
}

\medskip

\centerline{Table 1.}

1) For  $N$-steps we have~$x\in S$ and~$p\ge {1\over 15}$. Let us remind the calculation of profit of this step. The vertex~$x$ is adjacent to an alive leaf~$a$ of the tree~$F$. The  vertex~$a$ is not a leaf of the tree, obtained after adjoining~$F_0$ to~$F$, due to this the profit was  decreased by~${13\over 15}$. New dead leaves did not appear among the vertices of the tree~$F$, hence,  new leaves and new dead leaves of the obtained tree are leaves and dead leaves of the tree~$F_0$, calculated with the same  coefficients as in calculation of~$\alpha'(F_0)$.
Therefore, $\alpha'(F_0)=p+{13\over 15}\ge {14\over 15}$.

Clearly, $\N_G(a)\cap \N_G(x) =\varnothing$, hence by Remark~\ref{r1} we have~$a\in T$. 
Thus~$a$ is adjacent with three vertices~$b_1,b_2,b_3\in V(F)$, these vertices do not belong to the tree~$F_0$. 
We adjoin to~$F_0$ vertices~$a,b_1,b_2,b_3$ (see figure~\ref{fig9},\,1) and obtain as a result a new base tree~$F_1$. Performed operation gives us the profit~$3\cdot {13\over 15} - 4\cdot {2\over 5}= 1$, hence,~$\alpha'(F_1)\ge {29\over 15}+p\ge 2$ and~$\alpha(G)\ge2$.

\smallskip \goodbreak
2) {\it General algorithm of construction of a base tree for $M$-steps.}

\noindent In this case~$x\in T$.  
Let us remind the calculation of  profit of this step. The vertex~$x$ is adjacent to two alive leaves~$a_1,a_2$ of the tree~$F$. 
One of the vertices~$a_1,a_2$ is not a leaf of the tree, obtained after adjoining~$F_0$ to~$F$ (due to this the profit was  decreased by~${13\over 15}$), another becomes dead leaf (due to this the profit was  increased by~${2\over 15}$). All other new  leaves and new dead leaves of the obtained tree are leaves and dead leaves of the tree~$F_0$, calculated with the same  coefficients as in calculation of~$\alpha'(F_0)$. Therefore, $\alpha'(F_0)=p+{11\over 15}$.

We adjoin to~$F_0$ vertices~$a_1,a_2\in V(F)$  and obtain as a result a new base tree~$F_1$. Performed operation gives us  profit~$2\cdot ({13\over 15} -{2\over 5})={14\over 15}$, hence,~$\alpha'(F_1) \ge {25\over 15}+p$. 
If~$a_1,a_2\not \in T$, then profit increases by at least~$2\cdot{1\over 5}$ and we obtain~$\alpha'(F_1)\ge {31\over 15}$, that is enough.

Let~$a_1\in T$, then~$d_G(a_1)=4$. Note, that~$a_1$ is a leaf of the tree~$F$, hence, it is not adjacent to vertices of the set~$W$, except~$x$, thus it is adjacent to at least two different from~$a_2$ vertices of the tree~$F$. These vertices do not belong to the tree~$F_1$, we adjoin them to the tree and obtain a new tree~$F_2$.  If we have added more than two vertices, we have profit at least~$p(A2)+p(A1)\ge {8\over 15}$ (adjoining of two vertices is a step~$A2$, adjoining of the third vertex is a step~$A1$) 
and~$\alpha'(F_2)>2$. That is enough.
 
The only remaining case is when we have added exactly two vertices, let it be~$b_1,b_2$. Note, that in this case the vertices~$a_1$ and~$a_2$
are adjacent. We have performed a step~$A2$ with profit at least~${1\over 15}$, thus,  $\alpha'(F_2)\ge {26\over 15}+p$. 
If~$p\ge {3\over 15}$ we have~$\alpha'(F_2)\ge {29\over 15}$, that is enough by lemma~\ref{lf115}. For steps with~$p\le {3\over 15}$ we
consider two cases: $a_2$ is a dead  leaf  (see figure~\ref{fig9},\,2a) and an  alive leaf of the tree~$F_2$, respectively.

\begin{figure}[!ht]
	\centering
		\includegraphics[width=1\columnwidth, keepaspectratio]{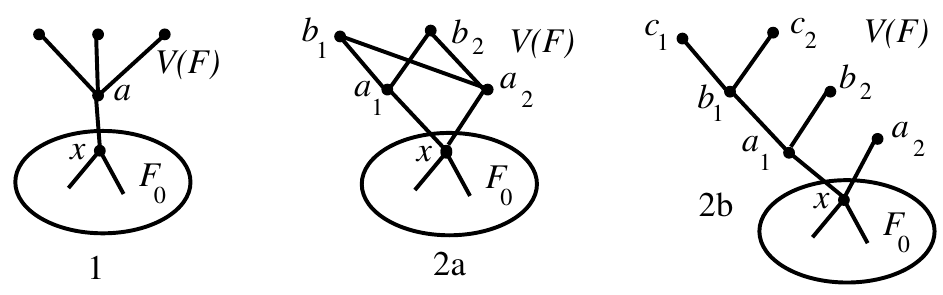}
     \caption{Construction of a base tree.}
	\label{fig9}
\end{figure}

\smallskip \goodbreak
\q{a}. {\it The vertex $a_2$ is a dead leaf of the tree~$F_2$.} 

\noindent That increases the profit by~$2\over 15$ and provide~$\alpha'(F_2) \ge {28\over 15}+p$. Both~$b_1$, $b_2$ are alive  leaves of the tree~$F_2$, otherwise profit increases by at least~$2\over 15$ and we have~$\alpha'(F_2)\ge 2$. If~$p\ge {1\over 15}$,  by lemma~\ref{lf115} we have~$\alpha(G)\ge 2$. The only remaining steps are~$M4.5.4$ and $M4.3$ with profit~0.

 \goodbreak 
\q{a1}. {\it Step $M4.5.4$.}

\noindent
Consider further construction of a spanning tree with the help of our algorithm. The tree~$F_2$ has exactly~7 alive leaves, hence, in the process of construction at least  7 alive leaves become dead. Let's look over Table~1: any step, decreasing the number of alive leaves, 
has  profit at least~$1\over 15$. Moreover, we can decrease the number of alive leaves with  profit exactly~$1\over 15$ only by~2 or by~5, i.e. less than by~7. Thus for killing of~7 alive leaves we obtain profit at least~$2\over 15$ and provide~$\alpha(G)\ge 2$.

\q{a2}. {\it Step $M4.3$.}

\noindent
Consider further construction of a spanning tree with the help of our algorithm. The tree~$F_2$ has exactly~4 alive leaves, hence, in the process of construction at least  7 alive leaves become dead. We need additional profit  at least~$2\over 15$. 
Killing of  alive leaves always gives us profit at least~$1\over 15$. The only possible number of alive leaves, which killing will not provide the profit~$2\over 15$ is~5. But to kill 5 alive leaves we must at first increase their number by exactly~1, and this operation provides profit at least~$1\over 15$. In any case we obtain~$\alpha(G)\ge 2$.

\smallskip
\q{b}. {\it The vertex~$a_2$ is an alive leaf of the tree~$F_2$.} 

\noindent Since~$a_2$ is not adjacent to~$a_1$, in this case at least one of the vertices~$b_1,b_2$ (let it be~$b_1$)  is not adjacent to~$a_2$. If $b_1\not\in T$,  then due to this the profit increases by~$1\over 5$ and we obtain~$\alpha'(F_2)\ge {29\over 15}$.
By lemma~\ref{lf115} that is enough for~$\alpha(G) \ge 2$. 

Let~$b_1\in T$. The vertices of the tree~$F_0$ are not adjacent to~$b_1\in V(F)$,  hence,~$b_1$ can be adjacent to at most two vertices of~$V(F_2)$: the vertex~$a_1$ and, maybe, the vertex~$b_2$. Therefore,~$b_1$ is adjacent to at least two vertices which do not belong to the tree~$F_2$. We adjoin all these vertices to the tree  $F_2$ and obtain a tree~$F_3$.  
If we have added more than two vertices, then, clearly,~$\alpha'(F_3)>2$. Hence, we have added exactly two vertices, let it be~$c_1,c_2$  
(see figure~\ref{fig9},\,2b). We consider leaves~$b_2$, $c_1$, $c_2$ of the tree~$F_3$ to be alive. If any of them is a dead leaf, it will be calculated at the end of construction with the help of a step~$Z0$.

We have performed a step~$A2$, obtained  profit~$1\over 15$ and~$\alpha'(F_3)\ge {27\over 15}+p$. 
If~$p\ge {2\over 15}$, then we have~$\alpha'\ge {29\over 15}$, that by lemma~\ref{lf115} provides~$\alpha(G)\ge 2$. 
Only the steps with profit less than~$2\over 15$ remain:  $M4.5.4$, $M4.3$ (with profit~0) and
$M4.2$,   $M4.5.3$, $M4.5.5$ (with profit~$1\over 15$). 

\smallskip \goodbreak
\q{b1}. {\it Step $M4.5.4$.}

\noindent
In this case there are exactly~9 alive leaves in the tree~$F_3$ and~$\alpha'(F_3)\ge {27\over 15}$. 
Consider further construction of a spanning tree with the help of our algorithm. 
At one step we can kill~1, 2, 3, 4 or~5 alive leaves (see table~1).
Killing at one step of any number of alive leaves, except~2 and~5, gives us profit at least~$2\over 15$. 
Hence, the only number of killed alive leaves, which is at least~9 and does not provide the profit at least~$3\over 15$ is 10 (one can kill~10 alive leaves with profit~$2\over 15$). But to kill 10 alive leaves we must at first increase their number by exactly~1, and this operation provides profit at least~$1\over 15$. In any case we obtain~$\alpha(G)\ge 2$.

\begin{rem}
\label{r454}
Now the cases of steps~$M4.5.4$ and~$N4.5.4$ in lemma~\ref{lmb} are completely analyzed. 
Assume, that  one have constructed a spanning tree~$T$ with~$\alpha(T)<2$ in the graph~$G$ with the help of our algorithm and it was performed 
one of the steps~$M4.5.4$ and~$N4.5.4$. Let~$F$ be a tree, constructed before this step.

Then at least one of vertices, adjoined on considered step, must be adjacent to the tree~$F$. Let us remember the details of this step: a leaf, adjacent to the tree~$F$, must be among two  leaves, adjoint to~$p_1$ at the end of the step (otherwise we would perform one of previous steps, see figure~\ref{fig6}). We call this leaf by~$q$.

If~$q\in T$ (in this case we call the steps~$M4.5.4.1$ and~$N4.5.4.1$), then~$q$ must be adjacent to  two leaves of the tree~$F$ (by remark~\ref{rl1}), that increases the number of dead leaves by two. Thus we obtain
$$ p(M4.5.4.1)\ge p(M4.5.4)+ 2\cdot {2\over 15}\ge {4\over 15} , \quad \Delta u=4, \quad  \Delta b=3, $$
$$ p(N4.5.4.1)\ge p(N4.5.4)+ 2\cdot {2\over 15} \ge {5\over 15}, \quad \Delta u=4, \quad \Delta b=2.$$

If~$q\notin T$ (in this case we call the steps~$M4.5.4.2$ and~$N4.5.4.2$), then we have one extra dead leaf. In addition, the profit increases by at least~$1\over 5$, since the cost of the vertex~$q$ decreases. In this case 
$$p(M4.5.4.2)\ge p(M4.5.4)+  {2\over 15} +{1\over 5} \ge  {5\over 15},\quad \Delta u=4,\quad \Delta b=2,$$
$$p(N4.5.4.2)\ge p(N4.5.4)+  {2\over 15} +{1\over 5} \ge  {6\over 15},\quad \Delta u=4,\quad\Delta b=1.$$

{\it Now we can claim, that any step, which increases the number of alive leaves, has profit at least~$1\over 15$. }
We shall take into account this property in forthcoming reasonings.
\end{rem}

Let us continue the proof of lemma~\ref{lmb}.

\smallskip \goodbreak
\q{b2}. {\it Step~$M4.3$.}

\noindent
In this case there are exactly~6 alive leaves in the tree~$F_3$ and~$\alpha'(F_3)\ge {27\over 15}$. 
Consider further construction of a spanning tree with the help of our algorithm. 
Killing  of any number of more than 5 alive leaves gives us  profit at least~$2\over 15$. 
Moreover, killing of exactly 6 alive leaves gives us profit at least~$3\over 15$.
Any step, which increases the number of alive leaves, has profit at least~$1\over 15$.
In any case we obtain~$\alpha(G)\ge 2$. 
\smallskip \goodbreak

\q{b3}. {\it Steps~$M4.2$, $M4.5.3$, $M4.5.5$.}

\noindent
In these cases we have~$\alpha'(F_3)\ge {28\over 15}$. 
We have in the tree~$F_3$ exactly~4 alive leaves for the step~$M4.5.5$ and exactly~7 alive leaves in other cases.
Hence we must kill at least~4 alive leaves. 
The only number of alive leaves, which killing will not provide profit~$2\over 15$ is~5 (profit~$1\over 15$). It is possible only for the step~$M4.5.5$,
if we increase the number of alive vertices exactly by~1, but this operation has profit at least~$1\over 15$.
In any case we obtain~$\alpha(G)\ge 2$. 
\end{proof}

Now there are significant transformations in our table of steps. 
 We exclude steps that provide~$\alpha(G)\ge 2$ by lemma~\ref{lmb} and  take into account remark~\ref{r454}.
All remaining steps and their parameters are introduced in table~2.

\medskip
{
\begin{tabular}{|@{\quad}c@{\quad}|@{\quad}c@{\quad}|@{\quad}c@{\quad}|}
\hline
Step &   $\Delta u-\Delta b$ &  $15\cdot$profit  \\ [2pt]
\hline
$A1$ &   1 &  $7$  \\
\hline
$A2$, \quad $A4$ &   1 & 1  \\
\hline
$A3$   &   2 &  2  \\
\hline
$M1$, \quad $M4.1.2$, \quad $M4.5.1$, \quad $N4.1.1$  &   0 &  3  \\
\hline
$N1$, \quad $N4.1.2$, \quad $N4.5.1$,  \quad $M4.5.4.1$ &   1 &  4  \\
\hline
$M2$, \quad $N4.5.4.1$, \quad $M4.5.4.2$ &   2 &  5  \\
\hline
$N2$, \quad $N4.5.4.2$   &   3 &  6  \\
\hline
$M3.1$   &   $-4$ &  $5$  \\
\hline
$N3.1$   &   $-3$ &  $6$  \\
\hline
$M3.2$   &   $-2$ &  $4$  \\
\hline
$N3.2$   &   $-1$ &  $5$  \\
\hline
$M4.1.1$, \quad  $Z0$  &   $-1$ &  $2$  \\
\hline
$Z1.1$   &   $-4$  &  $2$  \\
\hline
$Z1.2$  &   $-3$  &  $3$  \\
\hline
$Z2.1$    &   $-5$  &  $1$  \\
\hline
\end{tabular}
}

\smallskip

\centerline{Table 2.}

\smallskip

\begin{rem}
\label{rb7} Looking over table~2, one cal easily conclude:

1) any step  gives us profit at least~$1\over 15$;

2) a step or several steps, increasing the number of alive leaves by~2 or by~5, gives us profit at least~$2\over 15$;

3) any step, which preserves the number of alive leaves, gives us profit at least~$3\over 15$.
\end{rem}

Let us continue case analysis in construction of a  base tree.

\smallskip

\q{B5}. {\it There are two adjacent vertices~$a\in T$ and~ $b\in S$, such that~$\N_G(a)\cap \N_G(a')=\varnothing$.}

\noindent We begin with the base tree~$F'$ in which the vertices~$a$ and~$b$ are adjacent to each other and to all vertices from~$\N_G(a)\cap \N_G(b)$. Clearly,~$u(F')= 5$, $c_G(F')\le {1\over 5}+6\cdot{2\over5}={13\over 5} $
 and~$\alpha'(F')\ge 5\cdot {13\over 15}  - c_G(F') \ge {26 \over 15} $.

If any leaf of the tree~$F'$ does not belong to~$T$, then its cost decreases by~$1\over 5$,  hence, $\alpha'(F')$ increases by~$1\over 5$.
In this case we have~$\alpha'(F')\ge {29 \over 15}$, by lemma~\ref{lf115} that is enough. 

Therefore, the remaining case is when all leaves of the tree~$F'$ belong to~$T$, i.e., have degree~4. 
Let us  construct a spanning tree by our algorithm. Consider two cases.

\smallskip \goodbreak
\q{B5.1}. {\it In the process of construction the number of alive leaves was increased.}

\noindent In the beginning this number is equal to~5. Killing  of any number of alive leaves, more than~5, gives us  profit at least~$2\over 15$. Moreover,  profit can be equal to~$2\over 15$ only for~7 or~10 alive leaves. For another number of alive leaves we obtain at least~$3\over 15$ for killing and in addition at least~$1\over 15$ for increasing of the number of alive leafs, that provides~$\alpha(G)\ge 2$.

Let the number of alive leaves was increased to~7 or~10 (i.e. it was increased by~2 or by~5). 
By Remark~\ref{rb7}, for this increasing we have profit at least~$2\over 15$. After that we have additional~$2\over 15$ for killing of alive leaves. That provides~$\alpha(G)\ge 2$.

\smallskip \goodbreak
\q{B5.2}. {\it In the process of construction the number of alive leaves was not increased.}

\noindent Assume, that there was performed a step, which preserves the number of alive leaves and we obtained the tree~$F_1$. 
By Remark~\ref{rb7},  this step has profit at least~$3\over 15$, hence,~$\alpha'(F_1)\ge {29\over 15}$. As we know by Lemma~\ref{lf115}, 
that is enough for~$\alpha(G)\ge 2$.

Let us  consider the remaining case, when all performed steps have decreased the number of alive leave. 
We must kill~5 alive leaves of the tree~$F'$. Any way to do it, except the step~$Z2.1$, provides the profit at least~$4\over 15$ and~$\alpha(G) \ge 2$
(see table~2).

\begin{figure}[!ht]
	\centering
		\includegraphics[width=0.6\columnwidth, keepaspectratio]{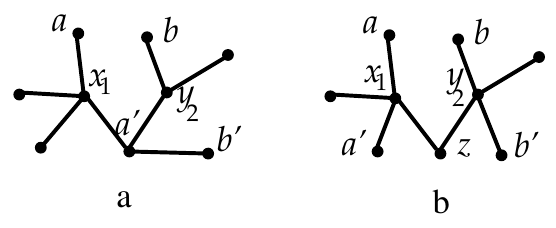}
     \caption{Case $B5.2$.}
	\label{fig10}
\end{figure}

Let it was performed the step~$Z2.1$, on which two adjacent vertices~$a'$ of degree~4 and~$b'$ of degree~3 were added. 
In this case our graph consists of~9 vertices: it contains two copies of the tree~$F'$:  with centers~$a,b$ and with centers~$a',b'$, and with five common leaves. Let these leaves are~$x_1,x_2,x_3\in \N_G(a)$ and~$y_1,y_2\in \N_G(b)$. Since $d_G(x_1)=d_G(x_2)=d_G(x_3)=d_G(y_1)=d_G(y_2)=4$, then~$G(\{x_1,x_2,x_3,y_1,y_2\})$ is a regular graph of degree~2, i.e. a cycle on five vertices. Hence, there exist two independent edges, connecting~$x_1,x_2,x_3$ with~$y_1,y_2$. Let these eges be~$x_1y_1$ and~$x_2y_2$.

Assume without loss of generality, that~$x_1$ and~$y_2$ are non-neighboring vertices of the 5-vertex cycle~~$G(\{x_1,x_2,x_3,y_1,y_2\})$. Then
$a,b,x_2,x_3,y_1 \in \N_G(x_1)\cup \N_G(y_2)$, moreover, one of the vertices~$x_2,x_3,y_1$ belongs to~$\N_G(x_1)\cap \N_G(y_2)$.

If $a'\in \N_G(x_1)\cap \N_G(y_2)$, then we construct a spanning tree in the following way. We connect~$a'$ with~$x_1$ and~$y_2$ 
and adjoin all other vertices to these three (see figure~\ref{fig10}a). Similarly in the case $b'\in \N_G(x_1)\cap \N_G(y_2)$. 

The remaining case is when one of the vertices~$a'$ and~$b'$ is adjacent to~$x_1$, and another~---  to~$y_2$. Then we connect~$x_1$ and~$y_2$ 
with a vertex from~$\N_G(x_1)\cap \N_G(y_2)$ (it is proved above, that such a vertex exists) and adjoin to~$x_1$ and~$y_2$ all other vertices  (see figure~\ref{fig10}b).  We obtain as a result in both cases a spanning tree of the graph~$G$ with~6 leaves. Note, that~$6>{2\over 5} \cdot 7 + {1\over 5} \cdot 2 +2$, i.e. the theorem  in this case is proved.

\begin{rem}
\label{rb6}
1) In remaining cases the steps~$Z2.1$ and~$A4$ are not performed since these steps need the configuration, considered in the case~$B5$. 

2) {\it Further on we assume, that any two adjacent vertices~$a,b\in V(G)$ have a common neighbor.} Assume the contrary, let~$\N_G(a)\cap \N_G(b)=\varnothing$. The case~$a,b\in S$ is impossible --- we would apply the reduction rule~$R2$. If~$a,b\in T$, then the graph~$G$ contains he configuration, considered in the case~$B1$. If one of the vertices~$a,b$ belongs to~$S$ and another to~$T$, then the graph  contains the configuration, considered in the case~$B5$.
\end{rem}

\goodbreak
\q{B6}. {\it The graph does not contain a vertex of degree~$4$.}

\noindent Hence~$G$ is a regular graph of degree~3. In~\cite{KW} it is proved, that  $u(G)\ge s\cdot {1\over 4} +2$ for such graph, 
whence  our theorem follows.

\begin{lem}
\label {lmb1} 
If~$\alpha(G)< 2$, then in each of steps~$M1$, $N1$, $M2$, $N2$,  $M3.1$, $N3.1$, $M3.2$, $N3.2$, $M4.1.1$, $N4.1.1$,
$M4.1.2$, $N4.1.2$,  $M4.5.1$,  $N4.5.1$, 
\linebreak
$M4.5.4.1$,  $N4.5.4.1$, $M4.5.4.2$,  $N4.5.4.2$ there must be an additional (with respect to parameters of this steps) dead leaf.
\end{lem}

\begin{proof}
Let us remind details of $MN$-steps. Let~$F$ be a tree before the step, $W=V(G)\setminus V(F)$. We have adjoined to the tree~$F$ a subtree (call it by~$F_0$) with a root~$x\in W$ and obtained after this step the tree~$F_1$. All vertices of~$W$, adjacent to~$V(F)$, are called by vertices of level~1.

Consider any step from our lemma. On this step one of  added to the tree~$F$ leaves~$v$ was adjacent to~$F$ (for the  steps
 $M1$, $N1$, $M2$, $N2$, $M4.5.5.1$,  $N4.5.4.1$, $M4.5.4.2$,  $N4.5.4.2$ it follows from Lemma~\ref{lmb}, for the steps~$M3.1$, $N3.1$, $M3.2$, $N3.2$ it was shown in detais of the step~$3$, for all other steps it follows from their description).

Note, that~$v\ne x$ (in all $MN$-steps the vertex~$x$ is not a leaf). Let us consider the ancestor~$w$ of the leaf~$v$ in the added tree~$F_0$.
By Remark~\ref{rb6} we have~$\N_G(v)\cap \N_G(w)\ne\varnothing$. Clearly,~$v,w\in W$.

Let~$a$ be a common neighbor of~$v$ and~$w$.   Clearly,~$a\not \in V(F)$, since a vertex of the tree~$F$ cannot be adjacent to two vertices of the set~$W$ by Remark~\ref{rl1}. By the construction, any vertex, adjacent to~$w$ belongs either to~$V(F)$ or to~$V(F_0)$. Hence,~$a\in V(F_0)$. Therefore, $v$ has two adjacent vertices~$w,a \in W$,  which belong to the tree~$F_1$, constructed after the step.
The vertex~$v\in W$ is adjacent to~$V(F)$, hence, it belongs to level~1. Then by Remark~\ref{rl1} the vertex~$v$ cannot be adjacent to more than two vertices from~$W$, thus the vertex~$v$ is a dead leaf of the tree~$F_1$. Clearly, it was not calculated in the parameters of the step.
\end{proof}

Before the last and most complicated case let us rewrite the table of parameters of the steps. We add dead leaves and update the profit for all steps of lemma~\ref{lmb1}.  In addition, we exclude the steps~$Z2.1$ and~$A4$, which are impossible due to remark~\ref{rb6}.
Updated parameters of the steps are presented in table~3.

\smallskip \goodbreak

\q{B7}. {\it The graph does not satisfy the condition of any previous case.}

\noindent Then there exists a vertex~$a$ of degree~$4$ in our graph. We connect~$a$ with its 4 neighbors and obtain the tree~$F'$ with~4 leaves~$v_1,v_2,v_3,v_4$ and~$\alpha'(F')\ge 4\cdot{13\over 15}-5\cdot{2\over 5}= {22\over 15}$. 
Let us continue the construction of a spanning tree by our algorithm and  
consider all performed steps, which have increased the number of alive leaves. We calculate the sum of  increases of the number of alive leaves over  all these steps and denote this sum by~$\ell$. 
Then the steps which decrease the number of alive leaves must kill~$\ell +4$  alive leaves. Consider several cases.

\smallskip \goodbreak
\q{B7.1}.  $\ell\ge 2$.

\noindent  It is easy to see from table~3, that adding of~$\ell$ alive leaves provides us profit at least~$\ell\over15$.
For~$\ell=2$ we must kill~6 alive leaves.  For this we obtain  profit at least~$6\over 15$ (see table~3), that provides~$\alpha(G) \ge 2$.
 
For~$\ell=3$ we must kill~7 alive leaves.  For this we obtain  profit at least~$5\over 15$ (see table~3), that provides~$\alpha(G) \ge 2$. 

For~$\ell\ge 4$ we must kill at least~8 alive leaves.  For this we obtain  profit at least~$4\over 15$ (see table~3), that also provides~$\alpha(G) \ge 2$.

\medskip

{
\begin{tabular}{|@{\quad}c@{\quad}|@{\quad}c@{\quad}|@{\quad}c@{\quad}|}
\hline
Step &   $\Delta u-\Delta b$ &  $15\cdot$profit  \\ [2pt]
\hline
$A1$ &   1 &  $7$  \\
\hline
$A2$ &   1 & 1  \\
\hline
$A3$   &   2 &  2  \\
\hline
$M1$, \quad $M4.1.2$, \quad $M4.5.1$, \quad $N4.1.1$  &   $-1$  &  5  \\
\hline
$N1$, \quad $N4.1.2$, \quad $N4.5.1$, \quad $M4.5.4.1$ &   0 &  6  \\
\hline
$M2$,\quad $N4.5.4.1$, \quad $M4.5.4.2$ &   1 &  7  \\
\hline
$N2$, \quad $N4.5.4.2$   &   2 &  8  \\
\hline
$M3.1$   &   $-5$ &  $7$  \\
\hline
$N3.1$   &   $-4$ &  $8$  \\
\hline
$M3.2$   &   $-3$ &  $5$  \\
\hline
$N3.2$   &   $-2$ &  $6$  \\
\hline
$M4.1.1$  &   $-2$ &  $4$  \\
\hline
$Z0$  &   $-1$ &  $2$  \\
\hline
$Z1.1$   &   $-4$  &  $2$  \\
\hline
$Z1.2$  &   $-3$  &  $3$  \\
\hline
\end{tabular}
}

\medskip
\centerline{Table~3.}

\smallskip

\q{B7.2}.  $\ell= 0$.

\noindent
It is easy to see from table~3, that the last step of algorithm has   profit at least~$2\over 15$. If it was performed a step, which preserves the number of alive leaves, it has profit at least~$6\over 15$. That is enough for~$\alpha(G) \ge 2$. 
Thus, only steps decreasing the number of alive leaves were performed.  It is easy to see from table~3, that there are three ways to kill 4 
alive leaves and not provide profit~$8\over 15$ (and~$\alpha(G)\ge 2$):  

---  step~$Z0$ together with step~$M3.2$ (sum of profits~$7\over 15$);

--- step~$Z0$ together with step~$Z1.2$ (sum of profits~$5\over 15$);

--- step~$Z1.1$ (profit~$2\over 15$).

Consider all these cases.

\smallskip
\q{B7.2.1}.  {\it Step~$Z0$ and step~$M3.2$ were performed.}

\noindent  Let we have performed the step~$M3.2$ and~$F$ be the tree before this step.
Then~$F$ must have at least 5 alive leafs (see figure~\ref{fig4}: 
two leaves of~$F$ must be adjacent to~$y_1$, one leaf~--- to~$y_2$ and, since~$x\in T$ for $M$-step, two leaves of~$F$ must be adjacent to~$x$). But in our case~$F=F'$ and this tree have only~4 leaves. We have a contradiction.

\smallskip \goodbreak
\q{B7.2.2}.  {\it   Step~$Z0$ and step~$Z1.2$ were performed.}

\noindent
We have profit~$5\over 15$, hence~$\alpha(G) \ge {27\over 15}$.
If at least one leaf of the tree~$F'$ does not belong to~$T$, then profit increases by~$3\over 15$, that provides~$\alpha \ge 2$. 
Thus, all leaves of~$F'$ belong to~$T$  and have degree~4  in the graph~$G$. 
Then there are exactly 6 vertices in the graph~$G$: five vertices of degree~4 and one vertex of degree~3 (added on the step~$Z1.2$). 
Clearly, this is impossible.

\smallskip \goodbreak
\q{B7.2.3}.  {\it Step~$Z1.1$ was performed.}

\noindent We have  profit~$2\over 15$, hence~$\alpha(G) \ge  {24\over 15}={8\over 5}$. 
The vertex, added on the  step~$Z1.1$ has degree~4. If at least two of vertices~$v_1,v_2,v_3,v_4$ do not belong to~$T$, 
the profit increases by~$6\over 15$, that provides~$\alpha(G) \ge 2$. If exactly one of these vertices does not belong to~$T$, then
the graph~$G$ contains five vertices of degree~4 and one vertex of degree~3, that is impossible.  Therefore, the only variant 
for~$\alpha(G)<2$ in our case is a regular graph of degree~4 on~6 vertices. Clearly, such graph  is unique --- it is~$C_6^2$, and this graph
is really an exclusion ($\alpha(C_6^2)={8\over 5}$).

\smallskip \goodbreak
\q{B7.3}.  $\ell= 1$.

\noindent That is we have performed exactly one step, increasing the number of alive leaves, and this number was increased exactly by~1. It is easy to see from  table~3, that it is either step~$A2$,  or we have obtained profit at least~$7\over 15$ and have constructed a tree~$F_1$ with~$\alpha'(F_1)\ge {29\over 15}$. In the last case by lemma~\ref{lf115} we have~$\alpha(G)\ge 2$ and finish the proof.

Thus we have done one step~$A2$ with profit~$1\over 15$ and have obtained the tree~$F_1$ with~$\alpha'(F_1)\ge {23\over 15}$. As above, 
if we have performed a step, which preserves the number of alive leaves, then~$\alpha(G)\ge 2$. 
Hence, the step~$A2$ is the only step except steps which decrease the number of alive leaves.
It is easy to see from table~3, that there is the only  way to kill 5 
alive leaves and not provide profit~$7\over 15$ (and~$\alpha(G)\ge 2$):   step~$Z0$ together with step~$Z1.1$ (sum of profits~$4\over 15$, we add a vertex of degree~4). These steps provide~$\alpha(G)\ge {27\over 15}$.

If at least one of the vertices~$v_1,v_2,v_3,v_4$ or two vertices, added on the step~$A2$, has degree~3, then the profit increases by~$1\over 5$. That provides~$\alpha(G)\ge 2$. 

Let us consider the last case --- when~$G$ is a regular graph of degree~$4$ on $8$ vertices. Clearly, such a graph is an exclusion if and only
if~$u(G)\le 5$. i.e.~$G$ does not have a  spanning tree with at least~6 leaves. It is easy to verify, that for $4$-regular graph~$G$ on $8$ vertices~$u(G)\le 5$ if and only if  {\it the neigborhoods of any two adjacent vertices have non-empty intersection}, i.e. {\it each edge belongs to a triangle}. 

\smallskip 
Let~$G$ be a 4-regular graph  on $8$ vertices, such that any its edge belongs to a triangle. Let us ensure, that up to isomorphism there are
two such graphs.

If~$G$ is a vertex  4-connected graph, we make use of the work~\cite{Mart} ---  it is proved there, that~$G$ is either a square of cycle or an edge graph of  4-cycle-connected cubic graph. The second case is impossible, since the number of vertices of such edge graph must be divisible by~3. The first case gives us the graph~$C_8^2$, which is really an exclusion.

Let~$G$ has a cutset~$R$, which consists of less than~4 vertices. It is easy to see, that if a set of~${4-k}$ vertices  separates in a 
$4$-regular graph a connected component~$H$, than~$v(H)\ge k+1$. Hence,  $|R|\ge 2$. 

If~$|R|=2$, then there is the only possibility: the cutset~$R$  must split the graph into exactly two connected components, each  component contains three vertices, and each of these 6 vertices must be adjacent to each of two vertices of the set~$R$. But then the  vertices of the set~$R$ have degree~6. We obtain a contradiction.

Let~$|R|=3$, $R=\{r_1,r_2,r_3\}$. Then one of connected components has two vertices (let these vertices are~$a_1,a_2$), and another component has three vertices ($b_1,b_2,b_3$). Clearly, $a_1$ and $a_2$ are adjacent and  each of them is adjacent to each of the vertices~$r_1,r_2,r_3$ (otherwise~$d_G(a_i)<4$). 
Hence, each of the vertices~$r_1,r_2,r_3$ is adjacent to not more than two of vertices~$b_1,b_2,b_3$, therefore, the sum of vertex degrees of the graph~$G(\{b_1,b_2,b_3\})$ is at least~6, i.\,e. this graph is complete. Consequently, each of the vertices~$b_1,b_2,b_3$ is adjacent to exactly two of vertices~$r_1,r_2,r_3$ and  each of vertices~$r_1,r_2,r_3$ is adjacent to exactly two of vertices~$b_1,b_2,b_3$, i.\,e. the vertices~$r_1,r_2,r_3$ are pairwise non-adjacent. Now it is clear, that there is only one such graph up to isomorphism --- the graph~$G_8$, shown on figure~1.

\subsection {Reduction and counterexamples}

{\it Let us prove, that if reduction rule~$R1$ or~$R2$ was applied to a graph~$G$, then the graph~$G$ is not an exclusion.}

As we know, applying of reduction rules~$R1$ and~$R2$ does not decrease~$\alpha(G)$.  Hence, it is enough to verify, that~$\alpha(G)\ge 2$ 
for a graph~$G$, which can be transformed to~$C_6^2$, $C_8^2$ or~$G_8$ by applying one reduction rule~$R1$ or~$R2$. Consider 6 cases.

\smallskip
\q1. {\it The graph~$G$ can be transformed to~$C_6^2$ by applying reduction rule~$R1$.} 

\noindent Let~$G$ can be transformed to the square of cycle~$a_1a_2a_3a_4a_5a_6$ by deleting a vertex~$w$ of degree~2 and adding  an edge connecting two vertices of~$\N_G(w)$. Without loss of generality we can consider two cases:  the added edge is~$a_1a_2$ or~$a_1a_3$. In both cases it is easy to construct a spanning tree with~5 leaves: see figures~\ref{fig11}a and~\ref{fig11}b.
Hence, $u(G)\ge 5> 6\cdot {2\over 5}+2$ and~$\alpha(G)> 2$.

\begin{figure}[!ht]
	\centering
		\includegraphics[width=1\columnwidth, keepaspectratio]{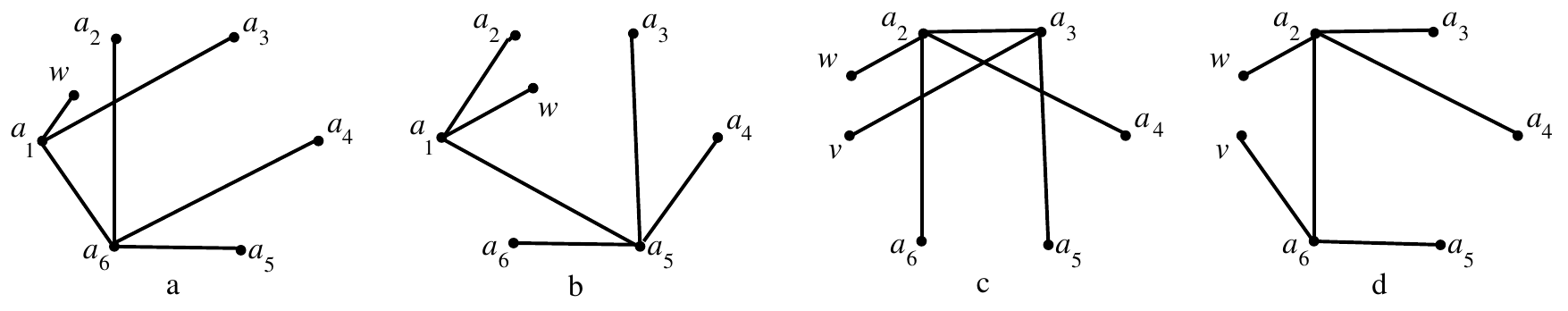}
     \caption{Reduction: case of~$C_6^2$.}
	\label{fig11}
\end{figure} 

\goodbreak
\q2. {\it The graph~$G$ can be transformed to~$C_6^2$ by applying reduction rule~$R2$.} 

\noindent  Let~$G$ can be transformed to the square of cycle~$a_1a_2a_3a_4a_5a_6$ by contracting an edge~$vw$, where~$d_G(v)=d_G(w)=3$.
Let the vertex~$a_1$ is the result of gluing of~$v$ and~$w$. Then~$v$ and~$w$ together are adjacent in~$G$ to~$a_2,a_3,a_6,a_5$. 
Without loss of generality we assume, that~$w$ is adjacent to~$a_2$ in the graph~$G$.
If~$a_3$ is adjacent in~$G$ to~$v$ then   we construct a spanning tree of the graph~$G$ to~$5$ leaves  as  on figure~\ref{fig11}c.
If~$a_3$ is adjacent in~$G$ to~$w$,  then~$a_6$ is adjacent in~$G$ with~$v$. In this case  a spanning tree of the graph~$G$  with~$5$ leaves  is shown on figure~\ref{fig11}d.
Hence, $u(G)\ge 5> 6\cdot {2\over 5}+2$ and~$\alpha(G)> 2$.

\goodbreak \smallskip
\q3. {\it The graph~$G$ can be transformed to~$C_8^2$ by applying reduction rule~$R1$.} 

\noindent 
Let~$G$ can be transformed to the square of cycle~$a_1a_2\dots a_8$ by deleting a vertex~$w$ of degree~2 and adding  an edge connecting two vertices of~$\N_G(w)$. Without loss of generality we can consider two cases:  the added edge is~$a_1a_2$ or~$a_1a_3$. In both cases it is easy to construct a spanning tree with~6 leaves: see figures~\ref{fig12}a and~\ref{fig12}b.
Hence,  $u(G)\ge 6> 8\cdot {2\over 5}+2$  and~$\alpha(G)> 2$.

\begin{figure}[!hb]
	\centering
		\includegraphics[width=1\columnwidth, keepaspectratio]{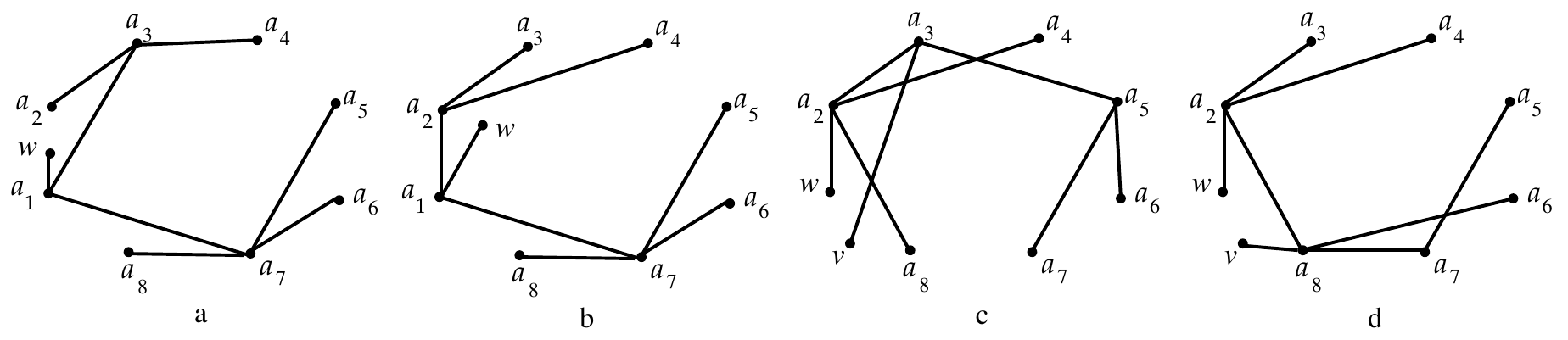}
     \caption{Reduction:  case of~$C_8^2$.}
	\label{fig12}
\end{figure} 

\q4. {\it The graph~$G$ can be transformed to~$C_8^2$ by applying reduction rule~$R2$.} 

\noindent
Let~$G$ can be transformed to the square of cycle~$a_1a_2\dots a_8$ by contracting an edge~$vw$, where~$d_G(v)=d_G(w)=3$.
Let the vertex~$a_1$ is the result of gluing of~$v$ and~$w$. Then~$v$ and~$w$ together are adjacent  in~$G$ to~$a_2,a_3,a_8,a_7$. 
Without loss of generality we assume, that~$w$ is adjacent to~$a_2$ in the graph~$G$.
If~$a_3$ is adjacent in~$G$ to~$v$ then we construct a spanning tree of the graph~$G$ with~$6$ leaves  as on figure~\ref{fig12}c.
If~$a_3$ is adjacent in~$G$ to~$w$,  then~$a_8$ is adjacent in~$G$ to~$v$. In this case  a spanning tree of the graph~$G$ with~$6$ leaves   is shown on figure~\ref{fig12}d. Hence,~$u(G)\ge 6$  and~$\alpha(G)> 2$.

\smallskip \goodbreak
\q5. {\it The graph~$G$ can be transformed to~$G_8$ by applying reduction rule~$R1$.}

\noindent  Let~$G$ can be transformed to the graph~$G_8$ by deleting a vertex~$w$ of degree~2 and adding  an edge connecting two vertices of~$\N_G(w)$. We use for~$G_8$ the same notations as on figure~\ref{fig1}.
By symmetry of the graph~$G_8$ it is enough to consider  four cases:  the added edge is~$a_1a_2$ (a spanning tree with~6 leaves is shown on figure~\ref{fig13}a), $b_1b_3$ (figure~\ref{fig13}b), $r_3b_3$ (figure~\ref{fig13}c) and~$r_3a_1$ (figure~\ref{fig13}d). 
Hence, in any case~$u(G)\ge 6$  and~$\alpha(G)> 2$.

\begin{figure}[!tb]
	\centering
		\includegraphics[width=\columnwidth, keepaspectratio]{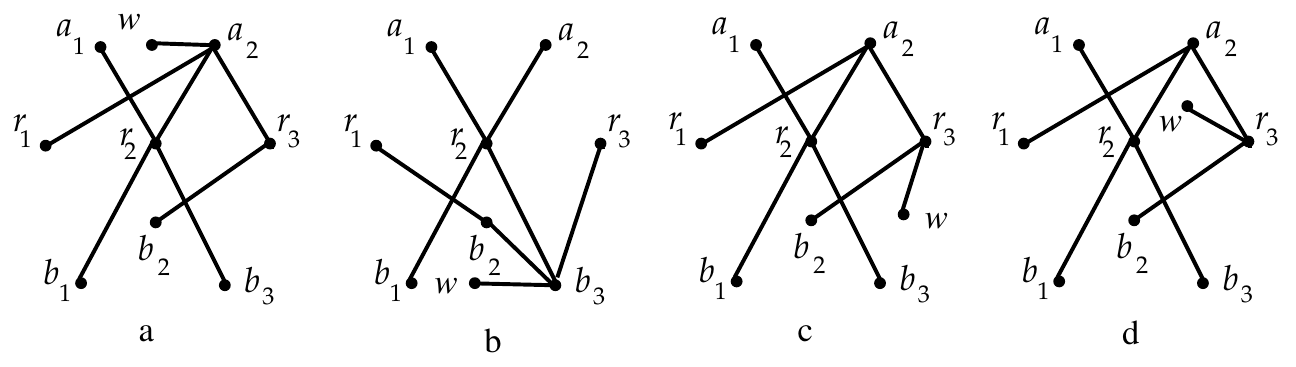}
     \caption{Reduction:  case of the graph~$G_8$ and rule~$R1$.}
	\label{fig13}
\end{figure}

\smallskip

\smallskip \goodbreak
\q6. {\it The graph~$G$ can be transformed to~$G_8$ by applying reduction rule~$R2$.} 

\noindent
Let~$G$ can be transformed to the graph~$G_8$ by contracting an edge~$vw$, where~$d_G(v)=d_G(w)=3$.
We use for~$G_8$ the notations of previous case. By symmetry of the graph~$G_8$ it is enough to consider  three cases:
the vertices~$v$ and~$w$ of the graph~$G$ can be contracted into one of the vertices~$a_1$, $b_1$, $r_1$. 

If it is~$a_1$, then~$\N_G(w)\cup \N_G(v) =\{a_2,r_1,r_2,r_3\}$.
Consider the tree~$F_1$, shown on figure~\ref{fig14}a. It has three non-pendant vertices~$a_2,r_2,r_3$  and each of vertices~$w$ and~$v$ 
is adjacent in the graph~$G$  to one of them. Thus,~$F_1$ can be transformed to a spanning tree of the graph~$G$ with~6 leaves.

\begin{figure}[!hb]
	\centering
		\includegraphics[width=0.85\columnwidth, keepaspectratio]{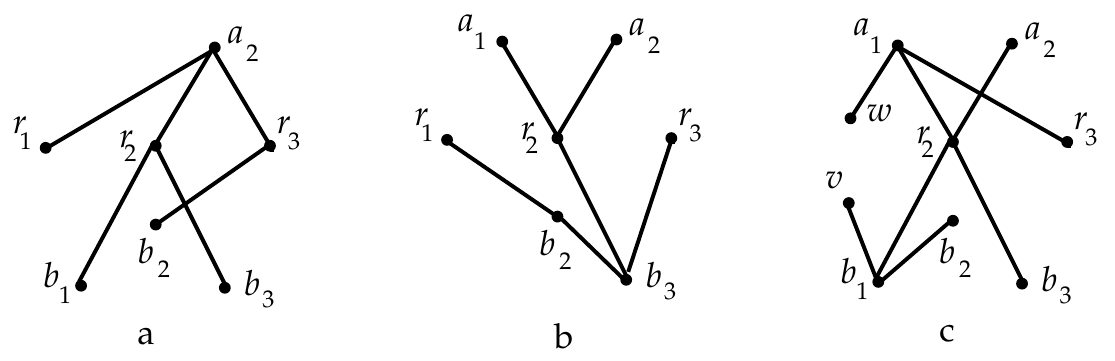}
     \caption{Reduction:  case of the graph~$G_8$ and rule~$R2$.}
	\label{fig14}
\end{figure} 

If~$v$ and~$w$ are contracted into the vertex~$b_1$, then~$\N_G(w)\cup \N_G(v) =\{b_2,b_3,r_1,r_2\}$.
Consider the tree~$F_2$, shown on figure~\ref{fig14}b. It has three non-pendant vertices~$b_2,b_3,r_2$  and each of vertices~$w$ and~$v$ 
is adjacent in the graph~$G$  to one of them. Thus,~$F_2$ can be transformed to a spanning tree of the graph~$G$ with~6 leaves.

If~$v$ and~$w$ are contracted into the vertex~$r_1$, then~$\N_G(w)\cup \N_G(v)=\{a_1,a_2,b_1,b_2\}$.
By symmetry it is enough to consider two cases:

--- $a_1,a_2\in \N_G(w)$, \quad $b_1,b_2\in \N_G(v)$; 

--- $a_1,b_2\in \N_G(w)$, \quad  $a_2,b_1\in \N_G(v)$. 

\noindent
In both cases we construct a spanning tree of the graph~$G$ with 6 leaves, as on figure~\ref{fig14}c.

Thus, in any case we have~$u(G)\ge 6$ and~$\alpha(G)> 2$.

\smallskip
Now we have completely proved Theorem~\ref{u34}.

\section{Extremal examples}

There are a lot of infinite series of graphs~$G$, containing~$s>0$ vertices of degree~3 and~$t>0$ vertices of degree more than~3, 
such that~$u(G)={2\over 5} t+ {1\over 4} s+2$. We introduce series of graphs, all vertices of which have degrees~3 and~4. 
Thus, these graphs are also counterexamples to the strong Linial's conjecture (see introduction).

Let us begin  construction of our graphs. Let~$D_i$ be a graph on vertex set~$x_i,y_i,z_i,v_i,a_i,b_i$, where the vertices~$x_i,y_i,z_i,v_i$  are pairwise adjacent, the vertex~$a_i$ is adjacent to~$x_i$ and~$y_i$, the vertex~$b_i$ is adjacent to~$z_i$ and~$v_i$. 
We make a cycle of such graphs~$D_1$,\dots, $D_n$ (where~$n>1$) and connect~$a_{i+1}$ with~$b_i$ (we set~$n+1=1$). 
The obtained graph we denote by~$H_n$  (see figure~\ref{fig15}). Clearly, $c(H_n)=2n\cdot {1\over 5}+ 4n\cdot {2\over 5} = 2n$.

\begin{figure}[!hb]
	\centering
		\includegraphics[width=0.7\columnwidth, keepaspectratio]{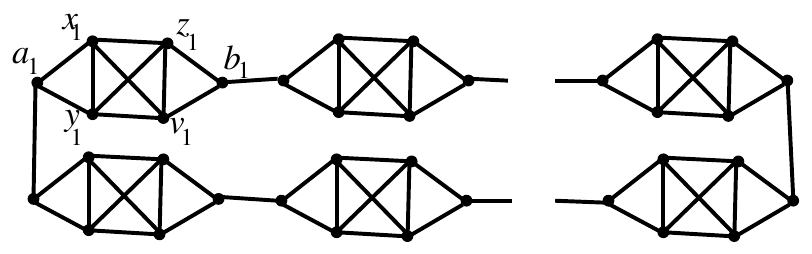}
     \caption{Extremal examples.}
	\label{fig15}
\end{figure}

Note, that the set of leaves of any spanning tree~$T$ of the graph~$H_n$ is not a cutset in~$H_n$.
With the help of this fact it is easy to see, that~$u(H_n)=2n+2=c(H_n)+2$.

\end{document}